\documentclass[12pt, leqno]{amsart}
\usepackage[OT2,T1]{fontenc}
\DeclareSymbolFont{cyrletters}{OT2}{wncyr}{m}{n}
\DeclareMathSymbol{\Sha}{\mathalpha}{cyrletters}{"58}

\usepackage{indentfirst}
\usepackage{amstext}
\usepackage{amsopn}
\usepackage{amsfonts}
\usepackage{latexsym}
\usepackage{amscd}
\usepackage{amssymb}
\usepackage{amsmath}
\usepackage[all,cmtip]{xy}
\usepackage{leftidx}
\usepackage{graphicx}
\usepackage{tikz}
\usepackage{ulem}

=\oddsidemargin \font\teneufm=eufm10 \font\seveneufm=eufm7
\font\fiveeufm=eufm5
\newfam\eufmfam
\textfont\eufmfam=\teneufm \scriptfont\eufmfam=\seveneufm
\scriptscriptfont\eufmfam=\fiveeufm

\let\goth\mathfrak
\def\cA{\mathcal A}
\def\cB{\mathcal B}

\def\cF{\mathcal F}

\def\cQ{\mathcal Q}

\def\GG{\mathbb{G}}
\def\NN{\mathbb{N}}
\def\FF{\mathbb{F}}

\def\gG{\goth G}

\def\gY{\goth Y}
\def\gX{\goth X}
\def\gZ{\goth Z}

\def\1{\mbox{\bf 1}}

\def\ug{\underline{g}}
\def\uG{\underline{G}}

\newcommand{\Mat}{{\operatorname{M}}}
\DeclareMathOperator{\supp}{supp} 
 \DeclareMathOperator{\Hom}{Hom}
\DeclareMathOperator{\Aut}{Aut}

\DeclareMathOperator{\GL}{\rm GL}
\DeclareMathOperator{\SL}{\rm SL}

\newcommand{\incl}[1][r]
{\ar@<-0.2pc>@{^(-}[#1] \ar@<+0.2pc>@{-}[#1]}

\newcommand{\ru}{{\mathcal R}_{u,k}} 

\newcommand{\ra}{{\mathcal R}_k} 

\newcommand{\rs}{{\mathcal R}_{s,k}} 

\newcommand{\rus}{{\mathcal R}_{us, k}} 

\newtheorem{theorem}{Theorem}[subsection]

\newtheorem{stheorem}{Theorem}[section]
\newtheorem{sclaim}[stheorem]{Claim}

\newtheorem{scorollary}[stheorem]{Corollary}
\newtheorem{slemma}[stheorem]{Lemma}
\newtheorem{sproposition}[stheorem]{Proposition}
\newtheorem{sremark}[stheorem]{Remark}
\newtheorem{sremarks}[stheorem]{Remarks}
\newtheorem{sexample}[stheorem]{Example}
\newtheorem{sexamples}[stheorem]{Examples}
\newtheorem{sdefinition}[stheorem]{Definition}

\theoremstyle{definition}
\newtheorem{remark}[theorem]{Remark}
\newtheorem{remarks}[theorem]{Remarks}

\numberwithin{equation}{section}

\newcounter{nc}
\renewcommand{\thenc}{{\rm(\roman{nc})}}
\newenvironment{romlist}%
{\begin{list}{\thenc}{
\usecounter{nc} 
\parsep=0pt
\setlength  \labelwidth{\leftmargin}
\addtolength\labelwidth{-\labelsep}
}
}{\end{list}}
\newcounter{nnc}
\renewcommand{\thennc}{{\rm(\alph{nnc})}}
{\begin{list}{\thennc}{
\usecounter{nnc}
\parsep=0pt
\setlength  \labelwidth{\leftmargin}
\addtolength\labelwidth{-\labelsep}
}
}{\end{list}} 
\newcommand{\pauseromlist}%
{\global\edef\savecount{\arabic{nc}}\end{romlist}}
\newcommand{\finpauseromlist}%
{\begin{romlist}\setcounter{nc}{\savecount}}

\newcounter{ctnum}
\renewcommand{\thectnum}{\textup{(\arabic{ctnum})}}
\newenvironment{numlist}%
{\begin{list}{\thectnum}{
\usecounter{ctnum} 
\parsep=0pt
\leftmargin=0pt%
\setlength{\itemindent}{\labelwidth}%
\addtolength{\itemindent}{\labelsep}%
}
}{\end{list}}
\def\NN{\mathbb{N}}
\def\ZZ{\mathbb{Z}}

\def\PP{\mathbb{P}}

\def\gE{\mathfrak{E}}
\def\gF{\mathfrak{F}}
\def\gG{\mathfrak{G}}

\def\gQ{\mathfrak{Q}}

\def\RR{\mathbb{R}}
\def\QQ{\mathbb{Q}}

\def\lra{\longrightarrow}

\def\ol{\overline}

\def\mod{\text{\rm mod}}

\def\2int{\mathop{2\int}\nolimits}

\def\Spec{\mathop{\rm Spec}\nolimits}

\def\Hom{\mathop{\rm Hom}\nolimits}

\def\Inf{\mathop{\rm Inf}\nolimits}

\def\mod{\mathop{\rm mod}\nolimits}
\def\supp{\mathop{\rm supp}\nolimits}

\def\Aut{\text{\rm{Aut}}}

\def\resp.{\mathop{\rm resp.}\nolimits}
\def\limproj{\mathop{\oalign{lim\cr
\hidewidth$\longleftarrow$\hidewidth\cr}}}
\def\limind{\mathop{\oalign{lim\cr
\hidewidth$\longrightarrow$\hidewidth\cr}}}

\def\Res{\mathop{\rm Res}\nolimits}
\def\res{\mathop{\rm res}\nolimits}

\def\lgr{\longrightarrow}

\font\math=cmmi10
\def\varpi{\hbox{\math\char'44}}

\def\simlgr{\buildrel\sim\over\lgr}

\def\pa{\S\kern.15em }

\def\un{\uppercase\expandafter{\romannumeral 1}}
\def\deux{\uppercase\expandafter{\romannumeral 2}}
\def\trois{\uppercase\expandafter{\romannumeral 3}}
\def\quatre{\uppercase\expandafter{\romannumeral 4}}
\def\cinq{\uppercase\expandafter{\romannumeral 5}}
\def\six{\uppercase\expandafter{\romannumeral 6}}

\def\hfl#1#2#3{\smash{\mathop{\hbox to#3{\rightarrowfill}}\limits
^{\scriptstyle#1}_{\scriptstyle#2}}}
\def\gfl#1#2#3{\smash{\mathop{\hbox to#3{\leftarrowfill}}\limits
^{\scriptstyle#1}_{\scriptstyle#2}}}

\title[Residues]{Residues on Affine Grassmannians }

\author{M. Florence}\address{Institut de 
Math\'ematiques de Jussieu, Universit\'e Paris 6, place Jussieu, 
75005 Paris, France}
\email{mathieu.florence@imj-prg.fr}

\author{P. Gille}\address{UMR 5208
Institut Camille Jordan - Universit\'e Claude Bernard Lyon 1
43 boulevard du 11 novembre 1918,
69622 Villeurbanne cedex - France 
}
\thanks{The authors are  supported  by the project ANR Geolie, ANR-15-CE
40-0012, (The French National Research Agency).
}
\email{gille@math.univ-lyon1.fr}

\date{\today}

\begin{document}

\maketitle

 \begin{abstract} Given a linear group $G$ over a field $k$, we define a notion of 
 index and residue of an element $g \in G\bigl(k(\!(t)\!) \bigr)$.
 The index $r(g)$ is a rational number  and the residue a group homomorphism
 $\res(g): \GG_a {\enskip or \enskip} \GG_m \to G$.
 This provides an alternative proof of  Gabber's theorem
 stating that $G$ has no subgroups isomorphic to $\GG_a$ or
 $\GG_m$ iff $G(k[[t]])= G(k(\!(t)\!))$.
 In the case of a reductive group, we offer an explicit connection 
 with the theory of affine grassmannians. 
 
\smallskip


\end{abstract}

{\small \tableofcontents }

\bigskip

\section{Introduction}

Let $k$ be a field and let $G$ be a linear algebraic $k$--group.
Our goal is to  associate in a quite elementary way 
to each element $g \in G(k(\!(t)\!)) \setminus G(k[[t]])$
an index $r(g) \in \QQ_{\geq 0}$ and a non-trivial  homomorphism
$\res(g): \GG_a \to G$ (resp.\,  a non-trivial homomorphism
$\res(g): \GG_m \to G$) if $r(g)>0$ (resp.\,  $r(g)=0$).
If $k$ is of characteristic zero (resp.\ $p>0$), 
we show that the index $r(g)$ is integral (resp.\ belongs to 
$\ZZ_{(p)}$), see Corollary  \ref{cor_main}.

This construction works actually over any ring $A$ 
with a closed subgroup scheme of $\SL_N$.
This permits to recover results by Gabber \cite{O} on the characterization of 
wound $k$--groups (i.e. $G$ does not contain $\GG_a$ or $\GG_m$)
and   also provides an  extension in the group scheme setting.
 This extends also to points $x\in X(k(\!(t)\!)) \setminus X(k[[t]])$ for $X$ a $G$--torsor;
 this is a key ingredient in the proof of the following result.

\begin{stheorem} (see Theorem \ref{thm_main2})
Let $X$ be a $G$--torsor such that $X(k(\!(t)\!)) \not = \emptyset$.
Then $X(k) \not=\emptyset$.
\end{stheorem}

For reductive groups, this statement is due 
to Bruhat-Tits (see \cite[I.1.2.1]{Gi} or  \cite[prop. 5.5]{Gu}).
The generalization of  that  statement over a ring
is known for $\GL_n$ and for tori according to recent results
by Bouthier-\v{C}esnavi\v{c}ius \cite[2.1.17, 3.1.7]{BC};
we generalize it as well  for wound closed subgroup schemes of $\SL_N$ 
(Cor. \ref{cor_wound_tors}) and for  $G$ commutative under
further assumptions (Th. \ref{thm_commutative}).
It is an open question beyond those cases.

Already over a field  it is an open question whether the statement does generalize
to homogeneous spaces; this is the case in characteristic $0$ according
to results by  the first author \cite{F}.

If $G$ is split reductive, we relate our construction of index and 
residue to the affine grassmannians $\cQ_G$.
The index provides a refinement of the stratification 
of the affine grasmannian $\cQ_G$ of $G$ (\S \ref{subsec_ind}).
In particular we show that an element \break $g \in G(k(\!(t)\!)) \setminus G(k[[t]])$
has index $0$ if and only if $g=g_1 \mu(t) g_2$ where
$g_1 \in G(k)$, $g_2 \in G(k[[t]])$ for some homomorphism
$\mu: \GG_m \to G$, see Proposition \ref{prop_fixed}.

\medskip

It is a pleasure to thank Ofer Gabber and Laurent Moret-Bailly for useful conversations.
We thank  also Simon Riche and  Xinwen Zhu for 
their expertise on affine grassmannians and Alexis Bouthier
for his reading.
Finallly we would like to thank the reviewers for their helpful comments.

\section{Indices}

\subsection{Notations and conventions}
If $r \in \QQ^\times$, the notation $r=m/n$ means
that $(m,n)=1$ with $n \geq 1$.
This extends to $0=0/1$.

For each ring $A$ (commutative, unital), we denote by $A^u=A[u]$
the ring of $A$-polynomials in the indeterminate $u$.
We denote by $A[[t]]$ the ring of power series
and define $A(\!(t)\!)=A[[t]][x]/(1-tx)$.
For each non-negative integer $n \geq 1$,
we define  $A[[t^{1/n}]]= A[[t]][y]/(y^n-t)$ and 
$A(\!(t^{1/n})\!)= A[[t]][x,y]/(y^n-t, 1-xy)$.
We have natural maps $A[[t^{1/n}]] \to A[[t^{1/mn}]]$
and  $A(\!(t^{1/n})\!) \to A(\!(t^{1/mn})\!)$ for $m \geq 1$.

If $r =m/n \in \QQ_{\geq 0}$, we put 
 $A_r= A^u[[t^{1/n}]]$ and
$\cA_r=A^u(\!(t^{1/n})\!)$.
We have a specialization homomorphism $j: A_r \to A^u$.

For each $r=m/n \in \QQ_{\geq 0}$, the assignment $t \to t(1+u t^r)$
defines ring homomorphisms $\sigma_r: A^u[[t]] \to A_r$; 
if $r>0$, its extends to $\sigma_r: A^u(\!(t)\!) \to \cA_r$.

Inverting $\lambda:= 1+u$, 
we come now to analogues $A^{u,+}= A[u][z]/ (1- (1+u)z))= A[\lambda, \lambda^{-1}]$.
 We have the variants $A^+_r = A^{u,+}[[t^{1/n}]]$;
$\cA^{u,+}_r=A^{u,+}(\!(t^{1/n})\!)$, $\sigma_r:  A^{u,+}(\!(t)\!) \to \cA^+_r$,
$t \mapsto  t(1+u t^r)$,  and  the specialization
$j^+: A^{u,+}_r \to A^{u,+}$ for all $r \in \QQ_{\geq 0}$.

\subsection{The ramification index}

Let $A$ be a ring and  let $G$ be an affine 
$A$--group scheme equipped with a closed embedding $\rho: G \to \SL_{N,A}$.

\begin{sproposition}\label{prop_index}
Let $g \in G\bigl(A(\!(t)\!)\bigr) \setminus G\bigl(A[[t]]\bigr)$. 
\begin{numlist}
 \item The set 
$$
\Sigma(g)= \Bigl\{ r \in \QQ_{>0} \, \mid g^{-1} \sigma_r(g) \in G(A_r) \Bigr\}
$$
is non-empty and let $r(g)$ be its lower bound in $\RR$.
Then $r(g) \in \QQ_{\geq 0}$ and \break
${\Sigma(g)= \QQ_{>0} \cap [r(g), + \infty[}$.

\item Assume that $r(g)>0$. Then  
 $j\bigl(g^{-1} \,  \sigma_{r(g)}(g)\bigr)$ belongs
to $G(A^{u}) \setminus G(A)$.

\item Assume that $r(g)=0$. Then 
$g^{-1} \sigma_0(g)  \in G(A^{u,+}[[t]])$ and 
 $j\bigl( g^{-1} \,  \sigma_0(g) \bigr)$ belongs
to $G(A^{u,+}) \setminus G(A)$.

\item Assume that $r(g)=m/n>0$. 
Let $\sigma: A[[t]] \to A_{r(g)}$ be a homomorphism such that 
$\sigma(t)=t(1 +  ut^{r} +  P_2(u) t^{m_2/n}+ \dots)$ 
with $m<m_2 <m_3< \dots < \dots$ and $P_i(u)\in A^u$ for $i \geq 2$.
Then $g^{-1} \sigma(g)$ belongs to  
$G(A_{r(g)})$ and $j\bigl( g^{-1} \sigma(g)\bigr)= 
j\bigl( g^{-1} \sigma_r(g)\bigr)$.

\end{numlist}
\end{sproposition}


\begin{proof}
(1) Clearly the statement reduces to the case of
$\SL_N$. 
Our assumption implies that 
$g=t^{-d} \ug$ with $d \geq 1$ and $\ug \in \Mat_N(A[[t]])
\setminus t \Mat_N(A[[t]])$. 
The number $-d$ is
called the gauge in $t$ of the matrix $g \in \Mat_N( A(\!(t)\!))$
and is denoted\footnote{The gauge is not multiplicative 
but is surmultiplicative, i.e. $V_t(g_1 g_2) \geq V_t(g_1)+ V_t(g_2)$.} by $V_t(g)$.
It follows that $\det(\ug)=t^{Nd}$.
For $r \in \QQ_{>0}$,  we have

\begin{equation}\label{formula0}
g^{-1} \sigma_r(g)= \frac{t^d}{t^d(1+u t^r)^d} \,\,  \ug^{-1} \,  \sigma_r(\ug)
= (1+u t^r)^{-d} \, \,  \ug^{-1} \,  \sigma_r(\ug).
\end{equation} 
We write  $\ug= \bigl(P_{i,j}\bigr)_{i,j=1,..,N}$ with $P_{i,j} \in A[[t]]$ 
and denote by $\Delta_{i,j} \in A[[t]]$  the minor of index $(i,j)$ of $\ug$.
We have $\ug^{-1}= \Bigl( t^{-Nd} \, \Delta_{i,j} \Bigr)_{i,j=1,..,N}$ so that the  
$(i,j)$--coefficient $C_{i,j,r}$ of $\ug^{-1} \,  \sigma_r(\ug)$
is
\begin{equation}\label{formula1}
C_{i,j,r}= t^{-Nd} \, \sum\limits_{k=1}^N
\Delta_{i,k}(t) \,  P_{k,j}(t (1+u t^r))\in A_r. 
\end{equation}
When $u=0$, $C_{i,j,r}$ specializes on  $\delta_{i,j}$ so that 
\begin{equation}\label{formula2}
C_{i,j,r} = \delta_{i,j} \, + \, t^{-Nd} \, \sum\limits_{k=1}^N \Delta_{i,k}(t) \,  
\bigl( P_{k,j}(t (1+u t^r)) - P_{k,j}(t) \bigr).
\end{equation}

\noindent We consider the identity

\begin{equation}\label{eq_coeff_epsilon}
\sum\limits_{k=1}^N \Delta_{i,k}(t) \,  
\bigl( P_{k,j}(t (1+ \epsilon)) - P_{k,j}(t) \bigr)
= \sum\limits_{a \geq 0, \\ \, b \geq 1} \, 
c^{a,b}_{i,j} \, t^{a} \,  \epsilon^{b}
\end{equation}
with $c^{a,b}_{i,j} \in A$.
Taking $\epsilon= u t^r$, we get

\begin{equation}\label{eq_coeff}
C_{i,j,r} = \delta_{i,j} \, + \,  t^{-Nd} \, \sum\limits_{a \geq 0, \\ \, b \geq 1} \, 
c^{a,b}_{i,j} \, t^{a+ r b} \,  u^{b}
\end{equation}
We consider the sets $\supp(i,j)= \bigl\{  (a,b) 
\, \mid \,  c^{a,b}_{i,j} \not =0  \, \}$
and $\supp(g)=\bigcup\limits_{(i,j)} \supp(i,j)$.

\begin{sclaim} \label{claim_supp} $\supp(g) \not = \emptyset$.
\end{sclaim}

If $\supp(g) = \emptyset$, then  $\ug=\sigma_r(\ug)$ and all coefficients
of $\ug$ belong to $A$ which contradicts the fact $\det(\ug)=t^{Nd}$.
The Claim is established and enables us to define the function 
$$
f_g(r)= \Inf\Bigl\{ -Nd +a+r b \,  \mid \, (a,b) \in  \supp(g)  \Bigr\}.
$$
We have $f_g(r) \geq r+ f_g(0)$ so that $f$ admits  positive  values
and $\Sigma(g)$ is not empty. Since $1+u t^r \in A_r^\times$, 
the set  $\Sigma(g)$ consists in the positive rational numbers $r$
such that $\ug^{-1} \,  \sigma_r(\ug)$ belongs to $\GL_N(A_r)$.
We get that
$$
\Sigma(g)= \bigl\{ r \in \QQ_{>0} \, \mid \, f_g(r) \geq 0 \bigr\}.
$$
By definition of $r(g)$, we have
$$
r(g)= \Inf\bigl\{ r \in \QQ_{> 0} \, \mid \, f_g(r) \geq 0 \bigr\} 
\in \RR_{\geq 0}.
$$
If $f_g(0) \geq 0$, then $r(g)=0$.
If $f_g(0) <0$, then there exists $a,b$ such that ${-Nd +a+r(g) b=0}$ 
whence $r(g) \in \QQ_{>0}$.
In both cases, we have  $\Sigma(g)= \QQ_{>0} \cap [r(g), + \infty[$.

\smallskip
 
\noindent (2) Along the proof of (1),
we have seen that then there exists $a,b$ such that ${-Nd +a+r b=0}$ and  $c^{a,b}_{i,j} \not =0$. 
Formula \eqref{eq_coeff} shows that $j\bigl(g^{-1} \, 
\sigma_r(g)\bigr) \not \in \SL_N(A)$.

\smallskip
 
\noindent (3) Once again, it is enough to consider the case of  $\SL_{N,A}$.
We recall the notation $A_1=A^u[[t]]$ and $\cA_1= \cA^u(\!(t)\!)$.
If $r(g)=0$, we have $a-Nd \geq 0$ for 
each $a$ occurring in  formula \eqref{eq_coeff} (more precisely such that $c^{a,b}_{i,j} \not =0$
for some $b$).
The point is that the computation of (1) works also for $r=0$.
It follows that 
$$
\ug^{-1} \sigma_0(\ug) - I_N \in u \Mat_N(A_1)
$$
so that $\ug^{-1} \sigma_0(\ug)  \in  \Mat_N\bigl(A_1 \bigr)$.
Taking into account the identity \eqref{eq_coeff} for $r=0$, we get 
\begin{equation}\label{eq_coeff0}
g^{-1} \sigma_0(g) = \lambda^{-d}  \ug^{-1} \sigma_0(\ug)
= \lambda^{-d} [I_N + u M_0(u) + t M_1(u)+ \dots]   \in \Mat_N\bigl( A[\lambda,\lambda^{-1}][[t]] \bigr) .
\end{equation}
with $M_i(u) \in \Mat_N(A[u])$ for $i=1,2,..$.
We have $j\bigl(g^{-1} \,  \sigma_0(g)\bigr)= \lambda^{-d}  \bigl( I_N + u M_0(u)\bigr)$. 
Assume that $j\bigl(g^{-1} \,  \sigma_0(g)\bigr)=M \in \SL_N(A)$. 
The formula \eqref{eq_coeff0} above reads
\begin{equation}\label{eq_coeff2}
g^{-1} \sigma_0(g)  = M
+ \lambda^{-d} \, \Bigl( t M_1(u)+ t^2 M_2(u)  + \dots \Bigr) \,  \in
\, \Mat_N\bigl( A[\lambda,\lambda^{-1}][[t]] \bigr) .
\end{equation} 
We consider the subring $B= \Bigr(A[\lambda][[t]]\Bigr)[\lambda^{-1}]$
of $A[\lambda, \lambda^{-1}][[t]]$ and its analogue \break $\cB= \Bigl(A[\lambda](\!(t)\!)\Bigr)[\lambda^{-1}]
\subset A[\lambda, \lambda^{-1}](\!(t)\!)$. 
The map $\sigma_0: A(\!(t)\!) \to A[\lambda, \lambda^{-1}](\!(t)\!)$,
$t \mapsto \lambda t$, factorizes through $\cB$.
It follows that  the equation \eqref{eq_coeff2} holds in $\Mat_N(\cB)$, i.e.
\begin{equation}\label{eq_coeff3}
g^{-1} g(\lambda t) = M
+ \lambda^{-d} [ t M_1(u)+ t^2 M_2(u)  + \dots]  
\in \Mat_N\bigl( \cB \bigr) .
\end{equation}
In other words we have
\begin{equation}\label{eq_coeff4}
g^{-1} g(\lambda t)  
= M+ \lambda^{-d} [ t M_1(\lambda-1)+ t^2 M_2(\lambda-1)  + \dots] 
\in \Mat_N\bigl( \cB \bigr) .
\end{equation}
The homomorphism $A[\lambda](\!(t)\!) \to A(\!(t)\!)$,
$\sum_{i \geq -L} P_i(\lambda) t^i \mapsto \sum_{i \geq -L} P_i(t) t^{d i} $
extends uniquely to a homomorphism  $\varphi: \cB \to A(\!(t)\!)$.
Specializing the equation \eqref{eq_coeff3} by $\varphi$  
yields that  $g^{-1}(t^d) g\bigl( t^{d+1}  \bigr)  \in 
 \Mat_N\bigl( A[[t]]\bigr)$ hence $$
g( t^{1+d} ) =  g(t^{d}) Q \hbox{\enskip with \enskip} Q \in \Mat_N\bigl(A[[t]]\bigr).
$$
It follows that $V_t( g( t^{1+d} ) )=  V_t( g(t^{d}) Q) \geq  V_t( g(t^{d})) +V_t(Q)
\geq V_t( g(t^{d}))$, so that 
$-d(d+1)  \geq - d^2$, this is a contradiction. 
We conclude that 
$j\bigl(g^{-1} \,  \sigma_0(g)\bigr) \not \in \SL_N(A)$. 

\smallskip
 
\noindent (4) We write $r=r(g)>0$ for short and continue to work with $\SL_{N,A}$.
We denote by $(\widetilde C_{i,j})$
the entries of $\ug^{-1} \, \sigma(\ug)$. 
Taking $\epsilon=u t^r+P_2(u) t^{m_2/r}+ 
\dots$, we have

\begin{eqnarray}\label{eq_coeff_bis} \nonumber 
\widetilde C_{i,j}  &=& \delta_{i,j} \, + \,   
t^{-Nd} \, \sum\limits_{a \geq 0,\, b \geq 1} \, 
c^{a,b}_{i,j} \, t^{a} \,  \bigl(u t^r + P_2(u) t^{m_2/n} + \dots \bigr)^b \\ \nonumber 
&=&  C_{i,j,r} \, + \,
t^{-Nd}\sum\limits_{a \geq 0,  \, b \geq 1} \, 
c^{a,b}_{i,j} \, t^{a} \, \bigl( u^b t^{b r} 
+ b u^{b-1} P_2(u) t^{\frac{(b-1)m+m_2}{n}} + \mbox{upper terms} \bigr).
\end{eqnarray} 
For each $(a,b) \in \supp(i,j)$, we have $-Nd+ a+ b r \geq 0$, so that
$-Nd+ a+ \frac{(b-1)m+m_2}{n} > 0$. 
Since $\ug^{-1} \, \sigma_r(\ug)$ belongs to $G(A_r)$, it follows that  
$\ug^{-1} \, \sigma(\ug)$ belongs to $G(A_r)$ 
and so does $g^{-1} \, \sigma(g)$. Furthermore the above computation
shows that $\ug^{-1} \, \sigma(\ug)= \ug^{-1} \, \sigma_r(\ug)$ 
in $\Mat_N(A_r)$ modulo $t^{\frac{m_2-m}{n}}$. Thus
$j\bigl( g^{-1} \sigma(g)\bigr)= 
j\bigl( g^{-1} \sigma_r(g)\bigr)$.
\end{proof}

\begin{sremark}\label{rem_F(u)}{\rm 
If $A$ is a field, by inspection of the proof, we see
that $$
r(g)= \Inf\Bigl\{ r \in \QQ_{>0} \, \mid g^{-1} \sigma_r(g) \in
G\bigl( A(u)[[t]] \bigr) \,\Bigr\}.
$$
 }
\end{sremark}

For later use, we record the following consequence of the proof
of Proposition \ref{prop_index}.

\begin{slemma}\label{lem_record} The assumptions are those 
of Proposition \ref{prop_index}. 
Let $M$ be a positive integer.

\begin{numlist}
 \item  Assume that  $r= r(g)=\frac{m}{n}$.
Then we have $$
 g^{-1} \, \sigma_s(g) \in \ker\Bigl( G\bigl(A^u[[t^{1/n}]]) \to 
 G\bigl(A^u[t^{1/n}]/t^{\frac{M}{n}}) \Bigr)
 $$
 for all $s=\frac{u}{n}$ with $u \geq m+M$.
 
 \item If $r(g)=0$ we have $$
 g^{-1} \sigma_s(g) \in \ker\Bigl( G\bigl(A^u[[t]]) \to 
 G\bigl(A^u[t]/t^M) \Bigr)
 $$
 for each  integers   $s \geq M$.
 \end{numlist}
\end{slemma}

\begin{proof}
 (1)   The $(i,j)$--coefficient of $g^{-1} \sigma_s(g)$
 reads 
 
\begin{equation}\label{eq_coeff_record}
D_{i,j,s} = (1+ut^s)^{-d} \, \Bigl( \delta_{i,j} \, + \,  
t^{-Nd} \, \sum\limits_{a \geq 0, \\ \, b \geq 1} \, 
c^{a,b}_{i,j} \, t^{a+ s b} \,  u^{b} \Bigr).
\end{equation}
 We write $s-r= \frac{v}{n}$ with $v \geq M$ and get 
 
 \begin{equation}\label{eq_coeff_record2}
D_{i,j,s} = (1+ u t^s)^{-d} \, \Bigl( \delta_{i,j} \, + \,  
 \sum\limits_{a \geq 0, \\ \, b \geq 1} \, 
c^{a,b}_{i,j} \, t^{-Nd + a+ r b} \,  t^{\frac{v b}{n}} \,  u^{b} \Bigr).
\end{equation}
 For each  non zero $c^{a,b}_{i,j}$, we have $-Nd + a+ r b \geq 0$ 
 so that $D_{i,j,s} - \delta_{i,j} \in t^{\frac{v}{n}} A^u[[t^{\frac{1}{n}}]]
 \subseteq t^{\frac{M}{n}} A^u[[t^{\frac{1}{n}}]]$.

 \smallskip
 
 \noindent (2) Let $(s,M)$  be a couple of integers satisfying  $s \geq M \geq 1$. the 
 $(i,j)$--coefficient of $g^{-1} \sigma_s(g)$
 reads 
 
 \begin{equation}\label{eq_coeff_record3}
D_{i,j,s} = (1+ u t^s)^{-d} \, \Bigl( \delta_{i,j} \, + \,  
 \sum\limits_{a \geq 0, \\ \, b \geq 1} \, 
c^{a,b}_{i,j} \, t^{-Nd + a} \,  t^{b s} \,  u^{b} \Bigr).
\end{equation}
 For each  non zero $c^{a,b}_{i,j}$, we have $-Nd + a \geq 0$ 
 so that $D_{i,j,s} - \delta_{i,j} \in t^{s} A^u[[t]] \subseteq t^{M} A^u[[t]]$.
 \end{proof}

\begin{sdefinition} {\rm Let $g \in G\bigl(A(\!(t)\!) \bigr)$.
If  $g \not \in G\bigl(A[[t]] \bigr)$, we define the {\it ramification index}
$r(g)$ as in Proposition 
\ref{prop_index}.
If  $g \in G\bigl(A[[t]] \bigr)$, we define $r(g)=-1$.
}
\end{sdefinition}

It is straightforward to check that the index does not depend of the
choice of the representation $\rho$.

\begin{slemma}\label{lem_0} We have
$$
\Bigl\{g \in G\bigl(A(\!(t)\!) \bigr) \, \mid \, r(g) \leq 0 \Bigr\} =  
\Bigl\{ g  \in G\bigl(A(\!(t)\!) \bigr)  \, \mid \,   g^{-1} \, g(\lambda t)
\in G\bigl( A[\lambda, \lambda^{-1}][[t]] \bigr)   \Bigr\} .
$$

\end{slemma}

\begin{proof} Let  $g \in G\bigl(A(\!(t)\!) \bigr)$. 
If $g  \in G\bigl(A[[t]] \bigr)$, it is obvious that $g$ belongs to 
the right-hand side. If $g  \not \in G\bigl(A[[t]] \bigr)$ and
 $r(g) \leq 0$, we have $r(g) =0$ and $g$ belongs to 
the right-hand side according to Proposition \ref{prop_index}.(3).

Conversely we assume that $g$ belongs to 
the right-hand side, that is  $g^{-1} \, g(\lambda t)
\in G\bigl( A[\lambda, \lambda^{-1}][[t]] \bigr)$.
We are given $r=m/n>0$. Since $1+ut^r$ is invertible in 
$A^u[[t^{1/n}]]$, we can make $\lambda=1+ut^r$ so that 
$g^{-1} \, \sigma_r(g) \in G\bigl( A^u[[t]] \bigr)$.
Since it holds for each rational $r>0$,
we get that $r(g)=0$ by definition of the index.
\end{proof}

\begin{slemma}\label{lem_sorites}
\begin{numlist}
 
\item The function  $g \to  r(g)$ is right $G(A[[t]])$--invariant
(resp.\ left $G(A)$-invariant)
 and is insensible to any injective base change $A \hookrightarrow A'$.
 
 \smallskip
 
\item Let $\phi: A \to B$ be a morphism of rings. Then  $r(g_B) \leq r(g)$.
 
 \smallskip
 
 \item Let $f:G \to H$ be a homomorphism between affine $A$--group 
schemes of finite type.

\smallskip

 \begin{romlist}
\item We have $r(f(g)) \leq r(g)$.
\smallskip

\item If $A$ is integral and $f$ is proper, we have $r(f(g))= r(g)$.
 \end{romlist}

 \smallskip

 \item  Let $G_1, G_2$ be affine $A$--group schemes of finite type
and consider the $A$--group scheme $G_1 \times_A G_2$.
For $g_i \in G_i(A(\!(t)\!))$ we have $r(g_1,g_2) \leq \Inf(r(g_1), r(g_2))$.
 
\smallskip

\item  Let $d$ be a non--negative integer and consider the map
$\phi_d: A(\!(t)\!) \to A(\!(T)\!)$ defined by $\phi_d(t)=T^d$.
We consider the map $\phi_{d,*}: G\bigl( A(\!(t)\!) \bigr) \to G\bigl( A(\!(T)\!) \bigr)$.

\smallskip

 \begin{romlist}
 
\item  If $d$ is not a zero divisor in $A$,
we have $r\bigl( \phi_{d,*}(g) \bigr)= d \, r(g)$.

\smallskip

\item If $A$ is of characteristic $p>0$ and $d=p^e$, we have 
$r\bigl( \phi_d(g) \bigr)=  r(g)$.
 
 \end{romlist}
 
\end{numlist}
\end{slemma}

\begin{proof}
(1), (2) and (3).(i) readily  follows  of the definition of $r(g)$.

\smallskip

\noindent (3)(ii): We can replace $A$ by its fraction field by (1),
allowing us to  assume that $A$ is  a field.
Since    $f:G \to H$ is assumed to be  proper,
 the valuative criterion of properness 
yields  $f^{-1}(H( A(u)[[t^r]]) \cap G\bigl(A(u)(\!(t^r)\!) \bigr)  = 
G\bigl(A(u)[[t^r]] \bigr)$ for each $r \in \QQ_{>0}$.
Then  Remark \ref{rem_F(u)} shows that 
$r(f(g))= r(g)$.

\smallskip

\noindent (4) follows from the definition of the index.

\smallskip

\noindent (5) We are given  $r=m/n \in \QQ_{>0}$. In case (i), the change 
$T \mapsto T(1+u T^r)$ induces $t=T^d \mapsto T^d(1+u T^r)^d= t( 1+ u \, d t^{\frac{m}{n d}} + \dots)$. 
According to  Proposition \ref{prop_index}.(4),
we get that $r(g) =
 \frac{r\bigl( \phi_{d,*}(g) \bigr)}{d}$. 
 
In  case (ii), we have 
$t=T^{p^e} \mapsto T^{p^e}(1+u T^r)^{p^e}= t( 1+ u^{p^e} t^{m/n} + \dots)$.
We get then $r\bigl( \phi_{d,*}(g) \bigr)= r(g)$.
 
\end{proof}

\begin{scorollary} Assume that $A$ is integral and  of characteristic $p>0$.
Let $e$ be a non-negative integer and 
let $F_e: G \to G^{(e)}$ be the $e$-iterated Frobenius morphism \cite[II.7.1.4]{DG}.
Then for each $g \in G\bigl( A(\!(t)\!) \bigr)$, we have 
$r\bigl(F_e(g) \bigr)= r(g)$. 
\end{scorollary}

\begin{proof}
Since $F_e$ is proper, this follows of Lemma \ref{lem_sorites}.(3).(ii). 
\end{proof}

\section{The residue}

Let $A$, $G$ and $g \in G(A(\!(t)\!))$ as in Proposition \ref{prop_index}.
If $r(g)>0$, we define the residue $\res(g)$ as the image of 
$g^{-1}\sigma_{r(g)}(g)$ by 
the homomorphism \break $j_*: G(A_r) \to G(A^u)=G(A[u])$.
We see it as an $A$--map $\res(g) :\GG_{a,A} =\Spec(A[u])\to \uG$
and will use sometimes the notation $\res(g)(u)$.

If $r(g)=0$,  we define the residue $\res(g)$ as the image of 
$g^{-1}\sigma_{0}(g)$ by 
the homomorphism $j_*: G(A^+_0) \to G(A^{u,+})=G(A[u, \frac{1}{1+u}])$.
Putting $\lambda=1+u$, we have $A^{u,+}=A[\lambda, \lambda^{-1}]$
so that we see the residue as an $A$--map 
$\res(g):\GG_{m,A}=\Spec(A[\lambda, \lambda^{-1}]) \to G$.
Similarly we use sometimes the notation $\res(g)(\lambda)$.

If $r(g)=-1$, i.e. $g \in G(A[[t]])$,  we put $\res(g)=1 \in G(A^u)$.
Again this does not depend of the choice of a representation.

\begin{sexamples}\label{example_basic}{\rm 
 (1) If $G=\GG_{m,A}$ and $g=\frac{1}{t^d}$, we have
 $$
 g^{-1} \, \sigma_0(g)= (1+u)^d=\lambda^d.
 $$
 In this case we have $r(g)=0$ and $\res(g)(\lambda)=\lambda^d$.
 
 \smallskip
 
 \noindent (2) If $A=k[\epsilon]$ is the ring of dual numbers and $G=\GG_{m,A}$,
 we consider the element $g = 1+ \frac{\epsilon}{t}$.
 Then $g^{-1} \sigma_r(g)=  (1- \frac{\epsilon}{t})  (1+ \frac{\epsilon}{t(1+ut^r)})
 = 1 -\epsilon u t^{r-1}$ so that $r(g)=1$.

 \smallskip
 
 \noindent (3)
 If $G=\GG_{a,A}$ and $g=\frac{1}{t^d}$ 
 with $d \in \ZZ_{\geq 1}$ invertible in $A$, we have
 $$
 g^{-1} \, \sigma_r(g)= \frac{-1}{t^d}+ \frac{1}{t^d(1 + u \, t^r)^d}
 =     \frac{-d \, u \, t^{r}+ \dots }{t^d(1 + u \, t^r)^d} 
 $$
 In this case we have $r(g)=d$ and $\res(g)(u)= -d \, u$.
 
 \smallskip
 
 \noindent (4) If $G=\GG_{a,A}$ and $g=\frac{1}{t^p}$ with $A$ of characteristic
 $p>0$  we have
 $$
 g^{-1} \, \sigma_r(g)= \frac{-1}{t^p}+ \frac{1}{t^p(1 + u t^r)^p}
= -\frac{u^p  \,  t^{rp}}{t^p(1 + u t^r)^p .
}
 $$
 In this case we have $r(g)=1$ and $\res(g)(u)= - u^p$.
 }
\end{sexamples}

\begin{sexample} \label{example_GL2} {\rm We consider the case $G=\GL_2$, and the element
$g= \begin{pmatrix}t^a& P(t) \\ 0 & t^d\end{pmatrix}$
with $P(t) \in A[t]$ and $a,d \in \ZZ$.
Putting $\epsilon = u t^{r}$, we have
$$
g^{-1} \sigma_r(g)= t^{-a-d} \, \begin{pmatrix}t^d& - P(t) \\ 0 & t^a \end{pmatrix}\, 
\begin{pmatrix} t^a(1+\epsilon)^a& P(t(1+\epsilon)) \\ 0 & t^d(1+\epsilon)^d\end{pmatrix}
=\begin{pmatrix}(1+\epsilon)^a& f \\ 0 & (1+\epsilon)^d\end{pmatrix}
$$
 with $$
 f= t^{-a-d} \Bigl( t^{d}P(t(1+\epsilon))- t^d(1+\epsilon)^d P(t) \Bigr)
 = t^{-a}\Bigl(  P(t(1+\epsilon)) - (1+\epsilon)^d P(t)   \Bigr).
 $$
\noindent (a) We take $a,d \geq 1$,  $P(t)=1$ and assume than $d$
is invertible in $A$. 
In this case, $P(t(1+\epsilon)) - (1+\epsilon)^d P(t)=
1- (1+ u t^r)^d$ so that  
$$
g^{-1} \sigma_r(g)= \begin{pmatrix}(1+ut^r)^a& -d \, t^{-a+r} u + \dots
\\ 0 & (1+ut^r)^d\end{pmatrix}
$$
It follows that  $r(g)=a$ and that 
$\res(g)= \begin{pmatrix}1& -d \, u \\ 0 & 1 \end{pmatrix}$.
 
 \smallskip
 
  \noindent (b) Assume that $A$ is an $\FF_p$--algebra and 
  take $d=p^s$ ($s \geq 1$) and $P(t)=t^{p^{ms}}$ with $m \geq 2$. Then
  $$
  f= t^{-a}\Bigl( t^{p^{ms}}(1+\epsilon^{p^{ms}}) - (1 + \epsilon^{p^s}) 
  t^{p^{ms}}\Bigr)
  =  t^{-a}\Bigl( t^{p^{ms}} (ut^r)^{p^{ms}} - t^{p^{ms}} (ut^r)^{p^{s}} \Bigr) 
  = - t^{-a + p^{ms} + r p^s} (u)^{p^{s}}+ \dots .
  $$
If  $a > p^{ms}$, we have $-a+p^{ms}+ r \, p^s =0$
 so that $r= \frac{a}{p^s} + p^{m(s-1)}$. In particular 
 $r$ can belong in $\ZZ[\frac{1}{p}] \setminus \ZZ$.
 
 \smallskip
 
 \noindent (c) For the multiplicative indeterminate 
 $\lambda$, we compute also 
 $$
g^{-1} g(\lambda t)= t^{-a-d} \, \begin{pmatrix}t^d& - P(t) \\ 0 & t^a \end{pmatrix}\, 
\begin{pmatrix} \lambda^a \, t^a & P(\lambda t) \\ 0 & \lambda^d \, t^d\end{pmatrix}
=\begin{pmatrix}\lambda^a& f \\ 0 & \lambda^d\end{pmatrix}
$$
 with $$
 f= t^{-a-d} \Bigl( t^{d}P( \lambda t)- \lambda^d t^d P(t) \Bigr)
 = t^{-a}\Bigl(  P(\lambda t) - \lambda^d P(t)   \Bigr).
 $$
 If $a \leq -1$, we have $r(g)=0$ and $\res(g)=\begin{pmatrix}\lambda^a & 0
 \\ 0 & \lambda^d \end{pmatrix}$. Furthermore 
 for $g'= g  \res(g)(t^{-1})$, we have
 $$
 {g'}^{-1} \, g'(\lambda t)= 
 \begin{pmatrix} t^a & 0
 \\ 0 & t^d \end{pmatrix}
 \begin{pmatrix}\lambda^a& f \\ 0 & \lambda^d\end{pmatrix}
 \begin{pmatrix}\lambda^{-a} t^{-a} & 0
 \\ 0 & \lambda^{-d} t^{-d} \end{pmatrix}
 = 
 \begin{pmatrix}1 &  f_1
 \\ 0 & 1 \end{pmatrix}
 $$
 with $f_1= t^{a-d} \lambda^{-d} f
 = \lambda^{-d} t^{-d}\Bigl(  P(\lambda t) - \lambda^d P(t)   \Bigr)$.
 For $a=-1$, $d=1$ and $P(t)=1$, we see that ${g'}^{-1}g'(\lambda t)$ does not belong
 in $\GL_2\bigl(  A[\lambda, \lambda^{-1}][[t]] \bigr)$.

Similarly for $g''=  \res(g)(t^{-1}) g$, we have
$$
 {g''}^{-1} \, g''(\lambda t)= 
 t^{-a-d} \, \begin{pmatrix}t^d& - P(t) \\ 0 & t^a \end{pmatrix}\, 
 \begin{pmatrix} \lambda^{-a}  & 0 \\ 0 & \lambda^{-d} \end{pmatrix}
\begin{pmatrix} \lambda^a \, t^a & P(\lambda t) \\ 0 & \lambda^d \, t^d\end{pmatrix}
=\begin{pmatrix}1& f_2 \\ 0 & 1\end{pmatrix}
$$ 
with $f_2=t^{-a} \Bigl[ \lambda^{-a} P(\lambda t) - P(t) \Bigr]$.
So for $a \leq -1$,  we see 
that ${g''}^{-1} \, g''(\lambda t)$  belongs
 to  $\GL_2\bigl(  A[[t]] \bigr)$.
 } 
\end{sexample}

For a group $\Gamma$, we recall the notation $^\sigma \! \tau= \sigma \tau \sigma^{-1}$
and $\tau^\sigma = \sigma^{-1} \tau \sigma$ for $\sigma, \tau \in \Gamma$.

\begin{slemma}\label{lem_residue}
Let $g \in G(A(\!(t)\!))$.
\begin{numlist}
 
 \item   Let $g_1 \in G(A)$ and $g_2 \in G(A[[t]])$.
 Then $\res(g_1 g g_2)=  \res(g)^{\overline{g_2}}$
 where  $\overline{g_2}$ stands for the specialization of 
 $g_2$ in $G(A)$. In particular we have 
  $\res(^{g_1}\!g)= \, ^{\overline g_1}\!\res(g)$.

\item Let  $A \to A'$ be a base change such that $r(g)=r(g_{A'(\!(t)\!)})$ (it holds for example
when $A$ injects in $A'$).
Then  $\res(g_{A'})= \res(g)_{A'}$.

\item  Assume that $A$ is integral and let $f:G \to H$ be a proper homomorphism between affine $A$--group 
schemes of finite type. We have $\res(f(g))= \res(g)$.

\item Let $d$ be a non--negative integer and consider the map
$\phi_d: A(\!(t)\!) \to A(\!(T)\!)$ defined by $\phi_d(t)=T^d$.
We consider the map $\phi_{d,*}: G\bigl( A(\!(t)\!) \bigr) \to G\bigl( A(\!(T)\!) \bigr)$.

 \begin{romlist} 
  
\item If $d$ is not a zero divisor in  $A$, we have $\res(\phi_{d,*}(g) \bigr)(u) = \res(g)(du)$ if $r>0$, or  $\res(\phi_{d,*}(g) \bigr)(\lambda) = \res(g)(\lambda^d)$ if $r=0$.

\smallskip

\item If $A$ is of characteristic $p>0$ and $d=p^e$, we have 
$\res\bigl( \phi_d(g) \bigr)= \res(g)(u^{p^e})$.

\end{romlist}
\end{numlist}

\end{slemma}

\begin{proof}
We write $r=r(g)=m/n$. 

\smallskip

\noindent (1) Since $\sigma_r(g_1)=g_1$, we have $(g_1 g g_2)^{-1} \, \sigma_r(g_1 g g_2)
= g_2^{-1}  (g^{-1} \sigma_r(g)) \, \sigma_r(g_2)
= g_2^{-1}  (g^{-1} \sigma_r(g)) \sigma_r(g_2)$.
When we specialize at $t=0$, we get $\res(g_1 g g_2)= \res(g)^{\overline{g_2}}$.
Assertions  (2) and (3)  follow of Lemma  \ref{lem_sorites}.

\smallskip

\noindent (4) We continue the proof of Lemma \ref{lem_sorites}.(5).
We have four cases to verify. 

\smallskip

\noindent{\it Case (i), $r>0$.} We have 
$r\bigl( \phi_{d,*}(g) \bigr)= d \, r(g)=\frac{d m}{n}$.
The change  $T \mapsto T(1+u T^r)$ induces
$t=T^d \mapsto T^d(1+u T^r)^d= \tau(t)= 
t( 1+ d \, u \, t^{\frac{m}{n d}} + \dots)$.
It follows that 
\[
 \phi_{d,*}( g^{-1} \,  \tau(g)) =
 \phi_{d,*}(g)^{-1} \, \sigma_{T,\frac{m}{n d}} \bigl( \phi_{d,*}(g) \bigr) 
 \in G( A^u[[T]])
\]
Proposition \ref{prop_index}.(4) yields that
$j( g^{-1} \,  \tau(g)) = j(g^{-1} \,  \sigma_{t,r}(g)) =
d\,\res(g) \in G(A^u)$.

\smallskip

\noindent{\it Case (i), $r=0$.} 
The change  $T \mapsto \lambda T$ induces
$t=T^d \mapsto \lambda^d t$.
It follows that 
\[
 \phi_{d,*}\bigl( g^{-1} \, g(\lambda t) \bigr) =
 \phi_{d,*}(g)^{-1} \,  g(\lambda^d T) 
 \res(g)(\lambda^d) \, ( 1 + \epsilon)
 \in G( A[\lambda^{\pm 1}[[T]])
\]
with $j( 1 + \epsilon)=1$. We conclude that 
$\res\bigl(\phi_{d,*}(g)\bigr)(\lambda)= \res(g)(\lambda^d)$.

\smallskip

\noindent{\it Case (ii), $r>0$.}  We have $r\bigl( \phi_{d,*}(g) \bigr)= r(g)$
and consider the base change
$t=T^{p^e} \mapsto T^{p^e}(1+u T^r)^{p^e}= \tau'(t)=t ( 1+ u^{p^e} t^{m/n} + \dots)$.
 It follows that 
\begin{equation}\label{formula_ram}
 \phi_{d,*}( g^{-1} \,  \tau'(g)) =
 \phi_{d,*}(g)^{-1} \, \sigma_{T,r} \bigl( \phi_{d,*}(g) \bigr) 
 \in G( A^u[[T]]).
\end{equation}
We put $v=u^{p^e}$ and consider $\sigma^v_r: A^v[[t]] \to A^v[[t^{1/n}]]$. 
By using Proposition \ref{prop_index}.(4) and the functoriality of 
the construction  $A^v \to A^u$, we have 
$j( g^{-1} \,  \tau'(g)) = j(g^{-1} \,  \sigma^v_{r}(g)) = 
\res(g)(v) =
\res(g)(u^{p^e}) \in G(A^u)$. 
By specializing  formula \eqref{formula_ram} at $T=0$, we get 
$\res(\phi_{d,*}(g) \bigr)(u) = \res(g)(u^{p^e})$.

\smallskip

\noindent{\it Case (ii), $r=0$.} It is similar. 
\end{proof}

\begin{stheorem} \label{thm_main}
\begin{numlist}
 
\item If $r(g)>0$, then $\res(g)$ is non-trivial homomorphism 
${\GG_{a,A} \to G}$.

\item If $r(g)=0$, then $\res(g)$  is a non-trivial homomorphism 
${\GG_{m,A} \to G}$.

\end{numlist}
\end{stheorem}

\begin{remark} {\rm In the case of an integral index, we  
 provide an  alternative proof for the homomorphism part of $(1)$ and $(2)$
 in Remark \ref{rem_alternative}.
 }
\end{remark}

 \smallskip

\begin{proof} We  can continue to work with $\SL_N$.
We write $r=r(g)=m/n$.

\smallskip

\noindent (1) 
We assume firstly that $n$ is invertible in $A$.
By developing the serie $(1+u t^r)^{1/n}$ in  $A_r$, 
we can extend $\sigma_r:  A^u[[t]] \to A_r$
to $\widetilde \sigma_r:  A_r \to A_r$.
The trick is to use the rings $A^{v_1,v_2}=A[v_1,v_2]$,
 $A^{v_1,v_2}[[t]]$ and  $A^{v_1,v_2}[[t^{1/n}]]$
and to define  morphisms $\tau_i: A^{v_1,v_2}[[t]] \to 
 A^{v_1,v_2}[[t^{1/n}]]$ ($i=1,2,3$) by 
$t \mapsto t (1+v_1 t^r)$, $t (1+v_2 t^r)$,
$t \bigl(1+ (v_1 +v_2) t^r \bigr)$ respectively.
These morphisms extend to morphisms
$\widetilde \tau_i: A^{v_1,v_2}[[t^{1/n}]] \to 
 A^{v_1,v_2}[[t^{1/n}]]$ for $i=1,2,3$.
We have the cocycle relation

\begin{equation}\label{eq_cocycle}
g^{-1} \, ({\widetilde \tau}_1 {\widetilde \tau}_2)(g) = g^{-1} \, 
{\widetilde \tau}_1(g)\, {\widetilde \tau}_1\bigr( g^{-1} \, {\widetilde \tau}_2(g) \bigr)
\end{equation}
inside $G\bigl(  A^{v_1,v_2}(\!(t^{1/n})\!) \bigr)$.
By using functoriality properties (Lemmas \ref{lem_sorites}.(1) 
and \ref{lem_residue}.(2)) we have
$g^{-1} \, {\widetilde \tau}_i(g) \in G\bigl( A^{v_1,v_2}[[t^{1/n}]] \bigr)$
with specialization $\res(g)(v_i)$. It follows that

\begin{equation}\label{eq_cocycle2}
g^{-1} \, ({\widetilde \tau}_1 {\widetilde \tau}_2)(g) = 
\res(g)(v_1) \, \res(g)(v_2)
\end{equation}
 inside $G\bigl( A^{v_1,v_2}[[t^{1/n}]] \bigr)$ modulo
the kernel of $G\bigl( A^{v_1,v_2}[[t^{1/n}]] \bigr) \to 
G\bigl( A^{v_1,v_2} \bigr)$.
On the other hand ${\widetilde \tau}_1 {\widetilde \tau}_2(t)=
t \bigl(1+ (v_1 +v_2) t^r + \hbox{upper terms}\bigr)$.
Proposition \ref{prop_index}.(3) applied to the ring $A[v_1]$ and 
$u=v_1+v_2$ shows that 
$g^{-1} \, ({\widetilde \tau}_1 {\widetilde \tau}_2)(g)
= \res(g)(v_1+v_2)$ inside $G\bigl( A^{v_1,v_2}[[t^{1/n}]] \bigr)$ modulo
the kernel of $G\bigl( A^{v_1,v_2}[[t^{1/n}]] \bigr) \to 
G\bigl( A^{v_1,v_2} \bigr)$.
We conclude that $\res(g)(v_1+v_2)=\res(g)(v_1) \times \res(g)(v_2)$.
Thus $\res(g)$ is an $A$--group homomorphism.

We explain now the refinement to the case $n=q \,  n'$ when 
$A$ is of characteristic $p>0$ and $(n',p)=1$ and $q=p^e$.
We consider  $A^u[[t]] \xrightarrow{\sigma_r} A^u[[t^{1/n}]] \to 
 A[u^{1/q}][[t^{1/n q}]]$
to  $\widetilde \sigma_r:   A^u[[t^{1/n}]] \to 
A[u^{1/q}][[t^{1/n q}]]$ by mapping
$t^{1/n}$ to the series $(1+ u^{1/ q} t^{r/ q})^{1/n'}$.

We extend then similarly the morphisms 
$A^{v_1^q,v_2^q}[[t]] \xrightarrow{\tau_i}
 A^{v_1^q,v_2^q}[[t^{1/n}]] \to  A^{v_1,v_2}[[t^{1/n q}]]$ ($i=1,2,3$) 
defined by 
$t \mapsto t (1+v_1^q t^r)$, $t (1+v_2^q t^r)$,
$t \bigl(1+ (v_1^q +v_2^q) t^r \bigr)$ 
in $\widetilde \tau_i: A^{v_1^q,v_2^q}[[t^{1/n}]] \to 
A^{v_1,v_2}[[t^{1/n q}]]$ for $i=1,2,3$.
The cocycle condition reads then 
$$
g^{-1} \, ({\widetilde \tau}_1 {\tau}_2)(g) = g^{-1} \, 
{\tau}_1(g)\, {\widetilde \tau}_1\bigr( g^{-1} \, {\tau}_2(g) \bigr).
$$
The same method yields 
$\res(g)(v_1^q+v_2^q)=\res(g)(v_1^q) \times \res(g)(v_2^q)$.
Thus $\res(g)$ is an $A$--group morphism.
 
\smallskip

\noindent (2) We have seen that 
$\res(g) \in   G(A[\lambda^{\pm 1}])
\setminus G(A)$ in Lemma \ref{lem_residue}.(4). We consider 
the ring $A[\lambda_1^{\pm 1},
\lambda_2^{\pm 1}]$
and the $A[\lambda_1,\lambda_2]$-automorphisms $\rho_i$
of $A[\lambda_1,\lambda_2](\!(t)\!)[\lambda_1^{-1}, \lambda_2^{-1}]$
defined respectively by $\rho_1(t)=  
\lambda_1 t$, $\rho_2(t) =\lambda_2 t$,
and $\rho_3(t) = \lambda_1 \lambda_2 t$.
Since $\rho_3= \rho_2 \circ \rho_1$, the cocycle relation 
$g^{-1} \, ({\rho}_1 {\rho}_2)(g) = g^{-1} \, 
{\rho}_1(g)\, {\rho}_1\bigr( g^{-1} \, {\rho}_2(g) \bigr)$ in 
$G\bigl( A[\lambda_1,\lambda_2](\!(t)\!)[\lambda_1^{-1}, \lambda_2^{-1}] \bigr)$
yields  $\res(g)(\lambda_1 \lambda_2)
=\res(g)(\lambda_1) \,  \res(g)(\lambda_2)$ as desired.
\end{proof}

This provides some control on the indices in the 
 integral case.

\begin{scorollary} \label{cor_main}
 Assume that $A$ is integral.
 Let $d \geq 1$ be an integer such that $t^d g \in M_N(A[[t]])$.

\begin{numlist}
 \item  If $A$ contains $\ZZ$, then $r(g) \in \ZZ$
 and $r(g) \leq Nd$.

\item If $A$ is a $\FF_p$-algebra for a prime
$p$, then $r(g) \in \ZZ[\frac{1}{p}]$
and there exists $s \geq 0$ such that $p^s r(g) \in \ZZ$
and $p^s r(g) \leq Nd$.

\end{numlist}
\end{scorollary}

\begin{proof} 
(1) and (2). If $r(g)=0$ the statements are clear so that we can assume that 
$r(g)>0$. We are allowed to replace $A$ by its fraction field $F$ according to Lemma
\ref{lem_residue}.(2).
We use now the decomposition $\SL_N\bigl( F(\!(t)\!) \bigr) =
B_N\bigl( F(\!(t)\!) \bigr) \, 
\SL_N\bigl( F[[t]] \bigr)$
where $B_N$ stands for the $F$--subgroup of upper triangular matrices \cite[4.4.3]{BT2}.
Lemma \ref{lem_residue}.(1)  permits to assume that $g \in  B_N\bigl( F(\!(t)\!) \bigr)$.
Coming back in the proof of Proposition \ref{prop_index}, we consider the
 coefficients of $g^{-1}\sigma_r(g)$  

\begin{equation}\label{eq_coeff_bis2}
D_{i,j,r} = (1+ut^r)^{-d}\, 
\Bigl( \delta_{i,j} \, + t^{-Nd} \, \sum\limits_{a \geq 0, \\ \, b \geq 1} \, 
c^{a,b}_{i,j} \, t^{a+ r b} \,  u^{b} \Bigr) .
\end{equation}

\noindent  We have $D_{i,j,r}=0$ if $j<i$.
We consider the non-empty set
$$
\Upsilon(g)= \Bigl\{  (i,j,a,b) \, \mid \, -Nd +a +r b=0  \hbox{\enskip and \enskip}
c^{a,b}_{i,j} \not = 0 \Bigr\}.
$$ 
\noindent 
It follows that the $(i,j)$-entry of $\res(g) \in B_N(k[u])$ is 
\begin{equation}\label{eq_coeff_ter}
c_{i,j} = \delta_{i,j} \, +  \, \sum\limits_{(i,j,a,b) \in \Upsilon(g)} \, 
c^{a,b}_{i,j} \,  u^{b}
\end{equation}

\noindent It follows that $c_{i,i}=1$ for each $i=1,..,N$, that is
$\res(g) \in U_N(k[u])$ where $U_N$ stands for the unipotent radical of $B_N$.
Let $(i,j,a,b)$ in $\Upsilon(g)$  such that  $i+j$ is minimal.
Since $\res(g)$ is a group homomorphism, it follows that 
$u \mapsto c_{i,j}^{a,b} u^b$ is a group homomorphism.

\smallskip

\noindent{\it Case of characteristic zero.}
In this case we have $b=1$. The equation $-Nd +a +r b=0$
yields that $r \in \ZZ$ and that $r=Nd - a \leq Nd$.

\smallskip

\noindent{\it Case of characteristic $p$.} It follows that $b$ is
a $p$--power, i.e.
$b=p^s$. Thus $r \in \ZZ[\frac{1}{p}]$ and 
$p^s r = Nd - a \leq Nd$.
\end{proof}

\begin{sremark}{\rm 
Corollary \ref{cor_main}.(1) is strenghtened by  Lemma \ref{lem_ind_var}.(1)
which shows that $r(g) \leq Nd$ without any condition on $A$.
 }
\end{sremark}

\subsection{Left index, left residue}

Given $g \in G(A(\!(t)\!))$, we define the left index and left
residue by $r^L(g)= r(g^{-1})$ and $\res^L(g)= \res(g^{-1})$.
If $G$ is commutative, we have  $r^L(g)= r(g)$ and $\res^L(g)= - \res(g)$.
This breaks in the non-commutative case.

\begin{sexample} {\rm In $\GL_2(k(\!(t)\!))$, we take $g= \begin{pmatrix}t& 1 
\\ 0 & t \end{pmatrix}$. We have seen that $r(g)=1$. We have 
$g^{-1}= \begin{pmatrix}t^{-1}& - t^2 
\\ 0 & t^{-1} \end{pmatrix}$
and
$$
g \, \sigma_r(g^{-1})= \begin{pmatrix}t& 1 
\\ 0 & t \end{pmatrix} \, \begin{pmatrix} t^{-1}(1-ut^r)^{-1}& - t^{-2}(1+ut^r)^{-2} 
\\ 0 & t^{-1}(1-ut^r)^{-1} \end{pmatrix}
= \begin{pmatrix}(1-ut^r)^{-1}& 0 
\\ 0 &(1+ut^r)^{-1} \end{pmatrix}
$$
so that $r(g^{-1})=0$.

}
\end{sexample}

\subsection{Extension to more affine group schemes}

Lemma \ref{lem_sorites}.(1) implies in particular that the 
index (and residue) is insensible to a  faithfully flat extension $A \to A'$ of rings.
Descent theory implies  that their definition extend to group schemes admitting
a closed embedding in $\SL_N$ after a  faithfully flat extension.
This applies in particular to group schemes of multiplicative type
(finitely presented) and to reductive group schemes.

From now on, we may then assume that $G$ 
is a closed $A$-subgroup of 
$\SL_N$ or that $G$ is affine of finite presentation
and there exists a  faithfully flat extension $A \to A'$
such that $G_{A'}$  is a closed $A'$-subgroup of 
some $\SL_{N',A'}$.

\subsection{Wound $A$--groups}

\begin{scorollary} \label{cor_wound}
 The following conditions are equivalent:

\begin{romlist}
 
\smallskip

\item $G\bigl(A[[t]] \bigr)=G\bigl( A(\!(t)\!) \bigr)$;

\smallskip

\item $G(A)=G\bigl( A[\lambda, \lambda^{-1} ]\bigr)$;
 
\smallskip

\item  $\Hom_{A-gr}(\GG_a,G)=0$ and $\Hom_{A-gr}(\GG_m,G)=1$.

\end{romlist}
\end{scorollary}

\begin{proof}
 $(i) \Longrightarrow (ii)$.
We have $G(A[[t]])=G\bigl( A(\!(t)\!) \bigr)$ and similarly
$G(A[[t^{-1}]])=G\bigl( A(\!(t^{-1})\!) \bigr)$.
If follows that $G(\PP^1_A)=
G\bigl( A[\lambda^{\pm 1}] \bigr)$. Since $G$ is affine of finite type over $A$,
we conclude that 
 $G(A)=G\bigl( A[\lambda^{\pm 1}] \bigr)$. 

\smallskip

\noindent $(ii) \Longrightarrow (iii)$. Let $h: \GG_a \to G$
be an $A$--homomorphism. 
Our assumption provides an element $g \in G(A)$ such that 
$h(u)=g$. Since $h(0)=1$, we 
conclude that $h$ is trivial.
Similarly we have $\Hom_{A-gr}(\GG_m,G)=1$.

\smallskip

\noindent $(iii) \Longrightarrow (i)$. This follows of Theorem \ref{thm_main}.
\end{proof}

\begin{sremark}{\rm If  $A=F$ is a field, 
Corollary \ref{cor_wound} was known
in the reductive case as the Bruhat-Tits-Rousseau's theorem 
\cite{P} (or \cite[prop. 3.5]{Gu}) and the general case is a consequence of Gabber's compactifications,
this will be discussed in \S \ref{section_gabber}.
}
\end{sremark}

\begin{sdefinition} {\rm
If the affine $A$--group scheme $G$ of finite type satisfies 
the conditions of  Corollary \ref{cor_wound}, we say that 
$G$ is $A$-{\it wound} ($A$-{\it ploy\'e} in French).
}
\end{sdefinition}

Our definition extends over 
any ring the case of  wound algebraic groups over a field \cite[5.1]{C}.

\section{Residues for  torsors}

\subsection{Definition}

\begin{sproposition} \label{prop_torsor} Let $X$ be a $G$--torsor over $A$. 
Let $x \in X\bigl(  A(\!(t)\!) \bigr)
\setminus  X\bigl( A[[t]] \bigr)$. For each $r \in \QQ_{\geq 0}$,
we denote by $g_r(x)$ the unique element of $G\bigl(  \cA_r \bigr)$
such that \break $\sigma_r(x)=x . \,  \,  g_r(x)$.

\smallskip

\begin{numlist}
 
\item The set 
$$
\Sigma(x)= \Bigl\{ r \in \QQ_{>0} \, \mid g_x(r) \in G(A_r) \Bigr\}
$$
is non-empty and let $r(x)$ be its lower bound in $\RR$.
Then $r(x) \in \QQ_{\geq 0}$ and \break
${\Sigma(x)= \QQ_{>0} \cap [r(g), + \infty[}$.

\smallskip

\item If $r(x) >0$, then $j\bigl(g_{r(x)}(x)\bigr) \in G(A^u)$
is a non-trivial homomorphism \break $\res(x) :\GG_a \to G$.

\smallskip

\item If $r(x) =0$, then $j\bigl(g_{r(x)}(x)\bigr) \in G(A^{u,+})$
is a non-trivial homomorphism \break $\res(x):\GG_m \to G$. 

\end{numlist}
\end{sproposition}

\begin{proof} In the case of a trivial $G$--torsor, the result follows
of Proposition \ref{prop_index} and of Theorem \ref{thm_main}. We shall reduce to this case. 
Let $B$ a flat cover of $A$ such that $X_B \cong G_{B}$.
Since $A$ embeds in $B$, we have $\Sigma(x) = \Sigma(x_B)$
so for (1) we are reduced to the case of $X_B$.
The statements (2) and (3) follow by descent from $B$ to $A$.
\end{proof}

Of course this extension of the residue to torsors satisfy 
the same functorialities as the residue for group schemes.

\subsection{The wound case}

\begin{scorollary} \label{cor_wound_tors} We assume that 
$G$ is $A$--wound.
Let $X$ be a $G$--torsor over $A$. Then $X$ is trivial 
if and only if  $X \times_{A} A(\!(t)\!)$ is trivial.  
\end{scorollary}

\begin{proof} The direct implication is obvious. We assume that 
 $X \times_A A(\!(t)\!)$ is trivial, that is $X\bigl( A(\!(t)\!) \bigr) \not = \emptyset$. 
If $X\bigl( A[[t]] \bigr) \not = \emptyset$, then $X(A) \not = \emptyset$
and $X$ is a trivial $G$--torsor. We can assume then that
 $X\bigl( A[[t]] \bigr)  = \emptyset$. We pick  $x \in X\bigl( A(\!(t)\!) \bigr)$
and Proposition \ref{prop_torsor} provides an element $\res(x)$
which is a non-trivial morphism $\GG_a \to G$ or a non-trivial morphism
$\GG_m \to G$. This contradicts our  assumption.
\end{proof}

\subsection{The commutative case}

\begin{stheorem}\label{thm_commutative}
We assume that $A$ is  an integral  $\QQ$-algebra, 
and that the $A$--group scheme $G$ is commutative.
Let $X$ be a $G$--torsor over $A$. Then $X$ is trivial 
if and only if  $X \times_{A} A(\!(t)\!)$ is trivial. 
\end{stheorem}

\begin{proof}
We assume that the $G$--torsor $X \times_{A} A(\!(t)\!)$ is trivial,
that is $X(A(\!(t)\!)) \not = \emptyset$. 
If $X(A)=\emptyset$, then the indices of points of $X(A(\!(t)\!))$
are all non-negative integers. Let $r$ be the minimal value of those
indices and  consider a point $x \in X(A(\!(t)\!))$ such that $r(x)=r$.

\noindent{\it Additive case: $r \geq 1$.}
We consider the point $y=x \, .\, \res(x)(-\frac{1}{r t^r})$.
According to Example \ref{example_basic}.(2), we have
\begin{eqnarray} \label{jo}
 - \frac{1}{t^r} + \sigma_r( \frac{1}{t^r})
& = &r u + r \epsilon
\end{eqnarray}
with $\epsilon \in t A^u[[t]]$.
Since $G$ is commutative, we have

\begin{eqnarray} \nonumber
 g_r(y)& =& g_r(x) \,  \bigl(\res(x)(-\frac{1}{r t^r})\bigr)^{-1}
 \, \sigma_r\bigl(\res(x)(-\frac{1}{r t^r})\bigr) \\ \nonumber
 &=& g_r(x) \,  \res(x)(-u t + \epsilon)  
\end{eqnarray}
by reporting the above identity \eqref{jo}.
We have $g_r(y) \in G\bigl(A^u[[t]] \bigr)$
so that $r(y) \leq r$. The reduction mod $t$ of
$g_r(y)$ is trivial by construction so that 
Proposition \ref{prop_torsor}.(2) shows that $r(y)<r$,
which contradicts the minimality of $r$.

\smallskip

\noindent{\it Multiplicative  case: $r =0$.} Similarly we use 
the homomorphism $\res(x): \GG_{m,A} \to G$ 
for constructing the point $y=x \, . \, \res(x)(t^{-1})^{-1}$.
Since $G$ is commutative, we have

\begin{eqnarray} \nonumber
 g_0(y)& =& g_0(x) \,  \res(x)(t^{-1}) \,
 \sigma_0\bigl(\res(x)(t^{-1}) \bigr)^{-1} \\ \nonumber
 & =& g_0(x) \,  \res(x)(\lambda)^{-1}  .
\end{eqnarray}
We have $g_0(y) \in G\bigl(A^+[[t]] \bigr)$
so that $r(y) \leq 0$. The reduction mod $t$ of
$g_r(y)$ is   $\res(x)(\lambda) \res(x)(\lambda^{-1})=1$ 
 so that  Proposition \ref{prop_torsor}.(2) shows that $r(y)<0$.
Thus $y \in X(A[[t]])$ and $X(A)$ is not empty.

\end{proof}

\section{The case of fields, II}
Let $k$ be a field.
A $k$--variety is a separated $k$--scheme of finite type.
If $X$ is a $k$--variety, we denote by $X^\dagger$
its largest geometrically reduced $k$--subscheme \cite[C.4.1]{CGP}.

A {\it $k$--compactification} of a $k$--variety $V$ 
is an open immersion  $j:V \to V^c$ where $V^c$ is a proper 
$k$--variety.

We use intensively the notions of the book \cite{CGP}, for example
the different radicals of a smooth  algebraic  $k$-group $H$.

\smallskip

$\ru(H)$: $k$--unipotent radicial, i.e. the largest smooth
connected unipotent normal $k$--subgroup of $H$.

\smallskip

$\ra(H)$: $k$--radical, i.e. i.e. the largest smooth
connected solvable normal $k$--subgroup of $H$.

\smallskip

$\rus(H)$: split unipotent $k$--radical, i.e. the largest smooth
connected $k$--split unipotent normal $k$--subgroup of $H$.

\smallskip

$\rs(H)$: split $k$--radical i.e. the largest smooth
connected $k$--split solvable normal $k$--subgroup of $H$.

\subsection{Pseudo-complete varieties}

We say that a  $k$-variety is {\it pseudo-complete}  if  $X(A)=X(K)$ for
each discretly valued   $k$-ring  $A$  with fraction 
fields $K$ and whose residue field is separable over  $k$
\cite[app. C.1.]{CGP}. To check this property, 
it is enough to consider the  case of a 
complete discretly valued $k$-ring
whose residue field is separably closed and separable over $k$
({\it ibid}, C.1.2.). In particular,
$X$ is  pseudo-complete if and only if 
$X_{k_s}$ is pseudo-complete.

We say that  $X$ is   $k$--pseudo-complete 
if  $X(A)=X(F)$ for each discretly valued  $F$-ring  $A$ 
of fraction field $F$ and of residue field $k$.
Similarly,  it is enough to consider the case 
 of $O=k[[t]]$ whose fraction field is denoted by $K=k(\!(t)\!)$.
 In particular an affine algebraic $k$--group $H$
 is $k$--pseudo-complete if and only if $H$ is $k$-wound according to Corollary \ref{cor_wound}.

\begin{slemma}\label{lem_fib}  Let  $H$ be a $k$--algebraic group. We assume 
that $H$ is $k$-pseudo-complete.
Let  $f: Y \to X$ be a  $H$--torsor where $Y$, $X$ 
are  $k$--varieties.
If $X$ is $k$--pseudo-complete,
then $Y$ is $k$--pseudo-complete. 
\end{slemma}

\begin{proof} We are given  $y \in Y(K)$ where $K=k(\!(t)\!)$. There exists
  a point  $x \in X(O)$ such that  $\pi(y)=x_K$.
  Let $\gZ= \pi^{-1}(x)$, this is a $H\times_k O$-torsor
and $y \in \gZ(K)$. Since the $O$--group scheme
$H\times_k O$ is $O$-wound, Corollary \ref{cor_wound_tors} shows that 
the $O$--torsor $\gZ$ is trivial, that is there is 
a  $H \times_k O$-equivariant isomorphism $\phi: \gZ \simlgr H \times_k O$.
Since $H$ is $k$-wound, we have $H(O)=H(K)$ hence
$\gY(O)= \gY(K)$. We conclude that $y \in Y(O)$.
\end{proof}

\subsection{Affine $k$--groups}

We recall the notations $K=k(\!(t)\!)$ and $O=k[[t]]$.

\begin{sproposition}\label{prop_parabolic} Let $G$ be an affine
algebraic $k$--group. Let $P$  
 be a minimal pseudo-parabolic  $k$--subgroup of  
$(G^\dagger)^0$.  We put  $R=\rs(G)$.

\smallskip

\begin{numlist}
 
\item The  quotient $Q=P/R$ is $k$--wound.

\smallskip

\item  For each  $G$-torsor $E$ over $k$  trivialized over  $k_s$, $E/R$ is 
$k$--pseudo-complete and  $E(k) \not = \emptyset$ if and only if 
$(E/R)(K) \not  = \emptyset$.

\smallskip 

\item  Let  $E$ be a  $G$--torsor. If   $E(K) \not =\emptyset$, then  
 $E(k) \not = \emptyset$.
 \end{numlist}

\end{sproposition}

\begin{sremarks}{\rm (a) 
In the perfect field case, 
statement (1) follows of a compactification result 
due to  Borel-Tits \cite[th. 8.2]{BoT}.

\smallskip

\noindent (b) Statement  (3) was known for groups 
of multiplicative type \cite[th. 4.1]{CTS2}. 
 }
\end{sremarks}

\begin{proof}
(1) We have   $\rs(P) \subseteq \ra(P)$ and $\ra(P)/\rs(P)$ 
is a  smooth connected solvable $k$--group \cite[th. 5.4]{C} which
fits in the exact sequence \[1 \to \ra(P)/\rs(P) \to P/\rs(P)
\to P/\ra(P) \to 1.\]
Let us show first that $P/\rs(P)$ does not contain 
any proper  central split $k$--subtorus.   Let $S \subset 
P/\rs(P)$ be a central  split $k$--subtorus;
its inverse image $M$ in $P$ is normal and  is an extension 
of $S$ by $\rs(P)$. Then $M$ is 
solvable $k$--split and 
by maximality of $\rs(P)$, it follows that  $S=1$.

If  $P/\rs(P)$ admits  a $k$--subgroup  $U$ isomorphic to  
$\GG_a$, $U$ cannot be a $k$--subgroup of the wound $k$--group
 $\ra(P)/\rs(P)$. It follows that the morphism $U \to P/\ra(P)$
is non trivial. Since any non--trivial quotient of  $\GG_a$ 
is isomorphic to  $\GG_a$ \cite[IV.2.1.1]{DG}, 
it follows that
$P/\ra(P)$ admits a $k$--subgroup  $V$ isomorphic to  $\GG_a$.
 We use now the fact that the quotient $P/\ra(P)$ is pseudo-reductive.
This implies that $V \subseteq \rus(Q')$ where $Q'$
is a  pseudo--parabolic $k$--subgroup of  $P/\ra(P)$ \cite[C.3.8]{CGP}.
The inverse image  $P'$ of $Q'$ in $P$ is  
a (proper) pseudo-parabolic $k$--subgroup of $P$ ({\it ibid}, 2.2.10),
and in the same time a pseudo-parabolic $k$--subgroup of 
 $(G^\dagger)^0$ ({\it ibid}, 3.5.5).
This leads to a contradiction with the minimality of  $P$.
Thus  $Q=P/\rs(P)$ is $k$-wound.

\smallskip

\noindent (2) \noindent{\it The smooth case.} Let  $E$ be a $G$--torsor
over $k$ and consider the quotient $Y= E/R$.
We claim that $X=E/P$ is pseudo-complete (and a fortiori $k$--pseudo-complete).
Since $E(k_s) \not = \emptyset$, we have 
$(E/P)_{k_s} \cong (G/P)_{k_s}$. 
We recall that  $(G/P)_{k_s}$ is  pseudo-complete \cite[C.1.6]{CGP},
so that  $(E/P)_{k_s}$ is pseudo-complete.  It follows that
$X=E/P$ is pseudo-complete.
We consider the  morphism $f: Y \to X= E/P$ which is a $Q=P/R$--torsor.
Since $Q$ is $k$--pseudo-complete, Lemma \ref{lem_fib} enables us to conclude that  
$Y$ is  $k$--pseudo-complete.

 \smallskip

\noindent{\it General case.} Since $E(k_s)\not = \emptyset$,
$E^\dagger$ is a  $G^\dagger$-torsor which satisfies
$E^\dagger \wedge^{G^\dagger} G = E$.
We consider the  morphism 
$G/R \to G/G^\dagger$ and its variant $q: E/R \to E/G^\dagger$.
We have  $(E/G^\dagger)(k)= (E/G^\dagger)(O)= (E/G^\dagger)(K)=\{ \bullet \}$.
The pre-image of this point is the  $k$--scheme $E^\dagger/R$ (actually 
$X^\dagger$). It follows that  $(E^\dagger/R)(K)= (E/R)(K)$. Appealing
to the smooth case, we get  $(E^\dagger/R)(O)= (E^\dagger/R)(K)$.
Thus $(E/R)(O)= (E/R)(K)$.

\smallskip

For the second property, we assume that  $X(k) \not = \emptyset$.
 The morphism   $E \to E/R=X$ is a $R$--torsor with  $R$ solvable and 
$k$--split. Since  $H^1(k,R)=1$ (by  d\'evissage from the cases of  $\GG_m$
 and $\GG_a$), we conclude that 
 $E(k) \not = \emptyset$.
 
\smallskip

\noindent (3) Let  $E$ be a $G$--torsor such that  $E(K) \not = \emptyset$.
Since  $K$ is  separable over  $k$, we have   $E^\dagger(K)=E(K)$ 
whence  $E^\dagger \not = \emptyset$. But  $E^\dagger$ is geometrically reduced
and is then generically smooth
\cite[32.25.7]{St}.  It follows that 
$E^\dagger(k_s) \not = \emptyset$ ({\it ibid}, 32.25.6).
A fortiori we have  $E(k_s) \not =\emptyset$ 
and the statement becomes a straightforward consequence of  (2).
\end{proof}

\subsection{Torsors over fields}
Let $G$ be a $k$--algebraic group. We continue with the notations
$K=k(\!(t)\!)$ and $O=k[[t]]$. Our purpose is to 
discuss the kernels of mappings
$$
a_G: H^1(O,G) \to H^1(K,G), \enskip
b_G: H^1(k,G) \to H^1(K,G), \enskip
c_G: H^1(O,G) \to H^1(k,G).
$$
If $G$ is smooth, the specialisation $c_G: H^1(O,G) \to H^1(k,G)$ 
is bijective according to Hensel's lemma \cite[XXIV.8.1]{SGA3}.

\begin{stheorem}\label{thm_main2} 
 The map $b_G:  H^1(k,G) \to H^1(K, G)$ is injective.
\end{stheorem}

\begin{slemma} \label{lem_proper}  
Let  $1 \to G_1 \to G_2 \to G_3 \to 1$ be 
an exact sequence of $k$--algebraic groups such that  $G_3(O)=G_3(K)$. 

\smallskip

\begin{numlist}
 \item We assume that  $G_3$ is smooth.
If  $\ker(a_{G_1}) \subseteq \ker(c_{G_1})$  and $\ker(a_{G_3}) 
\subseteq \ker(c_{G_3})$, then
 $\ker(a_{G_2}) \subseteq \ker(c_{G_2})$.

\item If  $\ker(b_{G_1})=1$ and $\ker(b_{G_3})=1$, then $\ker(b_{G_2})=1$.
\end{numlist}

\end{slemma}

\begin{proof}
 (1)  It is a classical ``d\'evissage'', see for example the proof 
of  \cite[Th. 2.1]{CTO}.
The above sequence gives rise to the exact commutative diagram of pointed sets 
$$
\begin{CD}
 G_3(k) @>{\varphi_k}>> H^1(k,G_1)  @>{\alpha_k}>> H^1(k,G_2) @>{\beta_k}>> H^1(k,G_3) \\
@AA{}A  @A{}AA @A{}AA @A{\wr}AA\\
 G_3(O) @>{\varphi_O}>> H^1(O,G_1)  @>{\alpha_O}>> H^1(O,G_2) @>{\beta_O}>> H^1(O,G_3) \\
@VV{\cong}V  @V{a_{G_1}}VV @V{a_{G_2}}VV @V{a_{G_3}}VV\\
 G_3(K) @>{\varphi_K}>> H^1(K,G_1) 
 @>{\alpha_K}>> H^1(K,G_2) @>{\beta_K}>> H^1(K,G_3) .\\
\end{CD}
$$
We are given an element $\gamma_2 \in \ker\bigl( H^1(O,G_2)  \to H^1(K,G_2) \bigr)$.
Since $G_3$ is smooth, $a_{G_3}$ has trivial kernel and hence there  exists a class
$\gamma_1 \in H^1(O,G_1)$ which maps on  $\gamma_2$ and such that 
$\gamma_{1,K}$ belongs to the image of the characteristic map $\varphi_K$.

We use now the right action of  $G_3(O)$ on  $H^1(O,G_1)$ 
and the fibers of  $\alpha_O$ 
are the orbits for that action \cite[III.3.3.3]{Gd}.
The same fact holds for  $G_3(K)$ and $\alpha_K$.
It follows that there exists $g \in G_3(K)$ 
such that  $\gamma_{1,K} \, . \, g_3 =1  \in H^1(K,G_1) $.
On the other hand, we have $g_3 \in G_3(O)$
hence $\gamma_{1,O} \, . \, g_3 \in \ker\bigl(  H^1(O,G_1) \to H^1(K,G_1) \bigr)$.
We can assume that  $\gamma_1 \in \ker(a_{G_1})$.
One of our assumption is that  $\ker(a_{G_1}) \subseteq  \ker(c_{G_1})$
Thus
$c_{G_2}( \gamma_2) =c_{G_2}( \alpha(\gamma_1)) = \alpha_k( c_{G_1}(\gamma_1))=1 \in 
 H^1(k,G_2)$.

The second fact is of the same vein.
 \end{proof}

We can proceed to the proof of Theorem \ref{thm_main2}.

\begin{proof} In the affine case, this is  
 Proposition \ref{prop_parabolic}.(3).
We assume $G$ connected. According to 
 Chevalley \cite[VI$_B$.12.5.(5)]{SGA3}, there is an exact sequence
$1 \to H \to G \to G/H \to 1$ where  $G/H$ is an abelian variety and  $H$ is affine.
By application of the valuative criterion of properness to $A$--torsors,
we have that   $\ker(b_A)=1$.
On the other hand, we have  $\ker(b_H)=1$, so that Lemma  \ref{lem_proper}.(2) 
shows that  $\ker(b_G)=1$.  To reach the non-connected case is 
of the same vein by using the exact sequence $1 \to G \to G \to G/G^0 \to 1$.
\end{proof}

\section{Link with Gabber's compactifications} \label{section_gabber}

\subsection{The statements}
The following two results have been announced by Gabber \cite{O}
and will not be used in the paper.
Let $k$ be a field. If $G$ is an algebraic $k$--group, 
we denote by $L'G$ its largest  smooth affine  connected  $k$--subgroup,
by $L_{\overline k}G$ the largest  largest  smooth affine  connected
subgroup of $G_{\overline{k}}$ and
by  $\overline{L}G$ the smallest $k$--subgroup such that 
$\overline{L}G \supset L_{\overline k}G$.
We consider the following property (due to Gabber)
$$
(*) \qquad \qquad  \hbox{All tori
of $G_{\overline k}$ are in $(G^\dagger)_{\overline k}$}.
\qquad \qquad\qquad \qquad\qquad \qquad \qquad \qquad\qquad 
$$

\begin{stheorem}\label{thm_gabber1}
 Let $G$ be a $k$-group of finite type and $P$ a pseudo-parabolic subgroup
of $L'G$. Then $G/P$
has an equivariant projective compactification, compatible with
$(G^\dagger)^0$, $L'G$,  $\overline{L}G$, with a $G$-linearized line bundle relatively ample
for $G/P \to  G/  \overline{L}G$ 
such that the boundary has no separable point and if $(*)$ holds there is no $k_s$-orbit
contained set-theoretically in the boundary.
\end{stheorem}

\begin{stheorem}\label{thm_gabber2}
 Let $G$ be a $k$-group of finite type. Then $G$ admits a projective
compactification $G \hookrightarrow G^c$ with a left action of $G$, 
right action of $G^\dagger$,  an invariant
ample effective divisor with support $G^c \setminus G$ if $G$ is affine, such that

\smallskip

\begin{numlist}
  \item For every separable extension $K$ of $k$ the following are equivalent:
  \begin{romlist}
   \item $G_K$ has a subgroup isomorphic to $\GG_a$ or $\GG_m$;
   \item $G(K) \not  =G^c(K)$;
   \item $\overline{L'G}(K) \not  = L'G(K)$;
   \item There is an $K$-orbit for the left action of $G$ on $G^c$
   admitting a separable point and contained in $G^c \setminus G$.
   \end{romlist}

\item For every separably closed separable extension $K$ of $k$, 
$\overline{G}(K) = G^\dagger(K) =
G(K) \overline{(L'G)}(K)$, and if $(*)$ holds every $K$-orbit of the action of $G$ on 
$G^c$ has a $K$-point.

\end{numlist}
\end{stheorem}

 \subsection{Refinement of Theorem \ref{thm_main2}}

 \begin{stheorem}\label{thm_star}
 Let $G$ be an affine algebraic $k$--group which satisfies
 the property $(*)$.
 Then $\ker\bigl( a_G\bigr)=1$.
 \end{stheorem}

\begin{proof} Once again we use the notations
$O=k[[t]]$ and $K=k(\!(t)\!)$. We consider the exact sequence of fppf  $k$-sheaves
$$
1 \to G^\dagger \to G \to G/G^\dagger \to 1. 
$$
We put $Q=  G/G^\dagger$, $\gG= G \times_k O$, $\gQ= Q \times_k O$.
According to Gabber \cite[th. 5.2]{GGMB}, $Q$ admits a $G$--equivariant 
compactification 
 $Q^c$ such that the boundary $\partial Q= Q^c \setminus Q$ 
 has no $G$-orbit over any separable field extension of $k$. Furthermore
 $Q^c$ carries a  $G$-linearized  line bundle. We put $\gQ^c= Q^c \times_k O$.
Let  $\gX$ be a  $\gG$--torsor such that  $\gX_K$ is trivial, that is
$\gX(K) \not = \emptyset$.
We consider the contracted products $\gZ= \gX \wedge^{G} \gQ$
 and $\gZ^c= \gX \wedge^{\gG} \gQ^c$ 
(which are  representable according to Lemma \ref{lem_descent}) over
$O$. We put $X= \gX_k$, $Z= \gZ_k$ and $Z^c= \gZ^c_k$.

\begin{sclaim} $Z(k)=Z^c(k)$.
\end{sclaim}
 
The argument is similar with that of the proof of \cite[Lemme 6.1]{GGMB}.
We consider the $k$--group scheme $G'=\Aut_G(X)$, it acts (on the left) 
on the morphism $X \to Z$. Let $z \in Z^c(k)$. According to \cite[2.3.2]{GGMB}, 
the $k$--orbit $T$ of $Z$ under $G'$ corresponds canonically to
 a $k$-orbit $T_0$ of $G$ on $Q^c$.
 By assumption, we have $T_0 \subset Q$, so that $T \subset Z$.
 In particular $z \in Z(k)$ and the Claim is proven.

We have compatible isomorphisms
$\gZ^c_K \simlgr \gQ^c \times_k K$ and $\gZ_K \simlgr \gQ \times_k K$.
Since $\{ \bullet \}= \gQ(K)= \gQ^c(K)$, we have that 
$\{ z_0 \}= \gZ(K)= \gZ^c(K)$.
The Claim implies that the specialization of $z_0$ in $Z^c(k)$
 belongs to $Z(k)$. It follows that $z_0 \in \gZ(O)$ 
 so that $\{ z_0 \}= \gZ(O)$.
Hence the  $\gG$--torsor $\gX$ admits a reduction $\gF$ 
to the subgroup $G^\dagger \times_k O$.
The exact commutative diagram
 \[
\xymatrix{
 \{\bullet \}= \gQ(O) \ar[r] \ar[d] &H^1(O,G^\dagger) \ar[d]^{a_{G^\dagger}} \ar[r]& H^1(O,G) \ar[d]^{a_G} \\
 \{\bullet \}= \gQ(K) \ar[r] \ar[r]& H^1(K,G^\dagger)  \ar[r]& H^1(K,G)
}
\] enables us to conclude that  $\gX$ is a trivial  $\gG$-torsor.
 
\end{proof}

\section{Applications and examples}

\subsection{Applications}
We shall apply the main result to nice torsors.

\begin{scorollary} Let  $G$ be an algebraic  $k$--group.

\smallskip

\begin{numlist}
 
\item We assume that  $G$ acts (on the left) on a $k$--variety $X$. 
Let  $x,x' \in X(k)$. Then  $x$ and $x'$ are $G(k)$--conjugated 
if and only if  $x_K$ and $x'_K$ are $G(K)$-conjugated. 

\smallskip

\item Let $H,H'$ be  $k$-subgroups of  $G$. Then  $H$ and $H'$ are 
$G(k)$--conjugated if and only if $H_K$ and $H'_K$ are $G(K)$--conjugated.
  
\end{numlist}
  
\end{scorollary}

\begin{proof} (1) We assume that  $x$,  $x'$ are $G(K)$--conjugated. 
We consider the transporter
$E= \bigl\{ g \in G \, \mid \,  g.x=x'\bigr\}$.
Since $E(K) \not = \emptyset$, $E$ is non-empty and is a  
 torsor under the stabilizer $G_x=\bigl\{ g \in G \, g.x=x\bigr\}$.
 Theorem \ref{thm_main2} yields that  $E(k) \not = \emptyset$.
Thus  $x$ and $x'$ are $G(k)$--conjugated.

\smallskip

\noindent (2) We assume that $H$ and $H'$ are  $G(K)$--conjugated.
We denote by $N=N_G(H)$ the normalizer of  $H$ 
in  $G$. We consider the strict transporter $T$ of $H$ to  $H'$ as defined in 
\cite[VI$_B$.6.2.4]{DG}.
Since $T(K)$ is non-empty,  $T$ is a  $N$--torsor.
 Theorem \ref{thm_main2} yields  that $T(k) \not = \emptyset$.
 Thus  $H$ and $H'$ are 
$G(k)$--conjugated.
\end{proof} 

Another useful statement is the following.

\begin{scorollary} Let $X$ be a $k$--variety  (resp.\ an algebraic $k$--group, etc..)
whose automorphism group is representable 
by an algebraic  $k$--group.
Let $X'$ be a  $k$--form of $X$. Then  $X$ and $X'$ are $k$-isomorphic
if and only if  $X_K$ and $X'_K$ are $K$-isomorphic.
\end{scorollary} 

\begin{proof} We apply Theorem \ref{thm_main}
to the $\Aut(X)$-torsor $\mathrm{Isom}(X,X')$.
\end{proof}

\smallskip

\subsection{Examples of $k$--groups such that  $\ker(a_G) \not =1$} \label{subsec_examples}

We assume that  $k$ is imperfect of  characteristic $p>0$ and 
we pick an element $a \in k\setminus k^p$.
 
 \smallskip
 
\noindent (a) Let $\alpha_p= \Spec(k[t]/t^p)$, this is a finite 
closed $k$-subgroup of $\GG_a$.
We denote by   $G_0= \mathrm{Aut}(\alpha_p)$ the  $k$--group 
(affine, algebraic) of  automorphisms
of the finite $k$--scheme $\alpha_p$. 
The pointed set $H^1(k,G_0)$ classifies the  $k$--forms 
of the $k$--algebra $A_0=k[x]/x^p$  \cite[III.5.1.10]{DG}.
The  $k$--algebra $A= k[x]/(x^p-a)$ is a 
$k$--form of $A_0$ and we denote by $G$ its automorphism group;
$G$ is a  $k$--form of $G_0$.

We consider now the   $O$--algebra $B= O[x]/(x^p - a t^p)$, 
which is an  $O$--form of $A \times_k O$.
This gives rise to a class $[B] \in H^1(O,G)$.
Since $A \otimes_k K \simlgr B \otimes_O K$, $[B]$ has trivial image 
in  $H^1(K,G)$.
On the other hand,  $a$ is not a $p$--power in  $B$
so that $A \otimes_k O$ and $B$ are not  $O$-isomorphic.
Thus  $[B] \not =1 \in H^1(O,G)$ and $\ker(a_G) \not = 1$.

\smallskip

\noindent (b) We shall construct an example of dimension  $1$ 
which arises from \cite[\S 7.1]{GGMB}.
We consider the action of the  $k$-group 
$G:= \GG_a \rtimes \GG_m$  (semi-direct product for the standard action of
$\GG_m$ on  $\GG_a$)
on the affine line $\mathbf{A}^1_k$ defined by 
$$
(x,y). z= x^p + y^p z \qquad (x \in \GG_a, \; y \in \GG_m\,, \;z \in \mathbf{A}^1).
$$
We observe that  $\mathbf{A}^1_k$ is a homogeneous $K$-space (on the left) 
under $G$; also  the  stabilizer of $0$ is the closed 
$k$-subgroup  $\alpha_p \rtimes_k \GG_m$.   
For an element $t_0 \in O$, we denote by $\gG_{t_0}$ 
the stabilizer of $t_0$, i.e.
$$
\gG_{t_0} = \Bigl\{ (x,y) \in G \times_k O \, \mid \, x^p + y^p t_0 = t_0 \Bigr\}.
$$
We are given $t_1 \in O$, the strict transporter from  $t_0$ to  $t_1$
is the  $(\gG_{t_0}, \gG_{t_1})$-bitorsor 
$$
\gE_{t_0, t_1} = \Bigl\{ \, (x,y) \in G \times_k O \, \mid \, x^p + y^p t_0=t_1 \Bigr\}.
$$
We consider the special cases  $t_0=a $ and $t_1= a t^p$
which is taylor made for having $\gE_{t_0, t_1}(K) \not = \emptyset$.
Since $a \not \in k^p$, one has
$\gE_{t_0, t_1}(k)=\emptyset$ and  a fortiori
$\gE_{t_0, t_1}(O)=\emptyset$. 

The point is that the $O$-scheme $\gG_a$ arises from 
the $k$--scheme $H= \bigl\{ \, (x,y) \in G\, \mid \, x^p + y^p a=a \bigr\}$.
We conclude that the map  $a_{H}: H^1(O,H) \to H^1(K,H)$ 
has a non trivial kernel. Furthermore we have  $\ker(a_H) \not \subset \ker(c_{H})$.

\begin{sremark}{\rm Theorem \ref{thm_star} shows that the 
$k$--group $G$ of Example (a) (resp.\ $H$ of Example (b)) does not 
satisfy property $(*)$. This fact can be checked directly.

\smallskip

\noindent (a) The $k$-group $G_0$ contains $\GG_m$ so that 
$G_{\overline k}$ contains $\GG_{m, \ol k}$.  On the other hand, we 
have $G(k_s)= \Aut(  k_s[x]/(x^p-a)=1$ so that $G^\dagger=1$.

\smallskip

\noindent (b) We have $H(k_s)=1$ so that $H^\dagger=1$.
On the other hand, we have $H_{\overline k} 
\cong (\alpha_p \rtimes \GG_m)_{\overline k}$ which contains 
$\GG_{m,\overline k}$.

 }
\end{sremark}

\section{A more advanced viewpoint}
 
In this section, our goal is to give an abstract exposition of the group structure
on automorphisms of $A[[t]]$  used in the second section.

 \subsection{Automorphisms of Laurent series and pro-group schemes}.
 
For each $w \geq 0$, we consider the affine 
 $\ZZ$--group scheme $J_w$ of automorphisms of 
 the sequence of rings 
 $$
 \ZZ[t]/t^{w+1} \to  \ZZ[t]/t^{w} \to \dots \to \ZZ
 $$
 For each ring $Z$, $J_w(R)$ consists of elements 
 $f_i \in  \Aut_{R-ring}( R[t]/t^{i})$ 
 for $i=0,\dots, w+1$ such that $f_{i}=f_{i+1} \mod t^i$ for $i=0,\dots,w$.
 We have transition maps $\pi^{w, w-1}: J_w \to J_{w-1}$ for each
 $w \geq 1$. The projective limit in the sense of \cite[\S 8]{EGA4}
 is denoted by  $J= \limproj J_w$. This is an affine $\ZZ$--group scheme
 whose coordinate ring is $\ZZ[J]= \limind \ZZ[J_w]$.
 We have projections $j_w: J \to J_w$ and we put $J^w= \ker( j_w)$.
 Since $J_0$ is trivial, we have $J^{0}=J$. The following 
 statement is straightforward.

 \smallskip
 
 \begin{slemma} Let $w \geq 1$
 and let $A$ be a ring. 
 \begin{numlist}
\item An  element of $J_w(A)$ is  given by 
$t \mapsto a_1 t + a_2 t^2 + \dots + a_w t^w$
with $a_1 \in A^\times$ and $a_2, \dots, a_n \in A$.

\item We have an exact sequence $1 \to \GG_a \to J_{w+1} \to J_w \to 1$.

\end{numlist}

\end{slemma}

In particular, 
 we have $J_1= \GG_m$ and the map $j^1 : J \to J_1=\GG_m$ is split
 by mapping a scalar $\lambda \in R^\times$ to $f_i= \times \lambda$.
 It follows that $J=J^1 \rtimes \GG_m$ where
 $J^1$ is a pro-unipotent $\ZZ$--group scheme.

 \smallskip
 
For a ring $R$, we have $R[[t]]= \limproj R[t]/ t^{w+1}$
so that $J(R)=  \limproj J_w(R)$ acts on $R[[t]]$.
We name  $J(R)$ the group of continuous automorphisms of $R[[t]]$.
An important thing is that $J(R)$ acts also on $A[[t]]$ for each
$R$--algebra $A$.

\subsection{Torsors and cocycles for Hochschild cohomology}
 
 We are given a ring $A$, an $A$--group scheme $G$ 
 equipped with a closed immersion $G \hookrightarrow \SL_N$,
 a  $G$--torsor $X$ and  a point 
 $x \in X\bigl(A(\!(t)\!) \bigr)$. 
 Of course an important case is $G$ itself.
 For each $A$--ring $B$ and each $\sigma \in J(B)$ we write
 $\sigma(x_B)=x_B. z_\sigma(x)$ with $z_\sigma(x) \in G\bigl(B(\!(t)\!) \bigr)$.
 
 For each $w \geq 0$ we denote by 
 $G_w=\prod_{A[t]/t^{w+1} \, / \, A} G_{A[t]/t^{w+1}}$ 
 the Weil restriction of $G$ with respect to the finite 
 $A$--algebra $A[t]/t^{w(x)+1}$.
 
 We denote by $L^+G$ the $A$--functor in groups $B \mapsto G\bigl( B[[t]] \bigr)$; 
there is a  natural map $L^+G \to G_w$ for each $w \geq 0$.
Similarly we denote by $LG$ the $A$--functor in groups
$B \mapsto G\bigl( B(\!(t)\!) \bigr)$.

 \begin{slemma}\label{lem_ww} There exists a smallest integer $w(x)$ such that 
 the restriction of $z(x)$ to $J^{w(x)}$ factorizes through $L^+G$.
 \end{slemma}

 \begin{proof} We have to prove that there exists $w \in \NN$ such that 
$z_\sigma(x) \in G\bigl(B[[t]] \bigr)$ for all $R$-algebras $B$
and all $\sigma \in J(B)$.
 Without loss of generality, we can replace $A$ by a faithfully flat extension so 
 we can assume that $X$ is a trivial $G$--torsor and $x=g \in G\bigl(A(\!(t)\!) \bigr)$.
We can replace then $G$ by $\SL_N$ and use the setting of the proof 
of Proposition \ref{prop_index}.
We write 
$g=t^{-d} \ug$ with $d \geq 0$ and $\ug \in \Mat_N(A[[t]])
\setminus t \Mat_N(A[[t]])$. It follows that $\det(\ug)=t^{Nd}$.
For an $A$--algebra $B$ and $\sigma \in J(B)$, 
we have
\begin{equation}\label{formula00}
g_{B[[t]]}^{-1} \sigma(g_{B[[t]]})= \bigl(\frac{\sigma(t)}{t} \bigr)^d \,\,  
\ug_{B[[t]]}^{-1} \,  \sigma(\ug_{B[[t]]}).
\end{equation}

\noindent We write  $\ug= \bigl(P_{i,j}\bigr)_{i,j=1,..,N}$ with $P_{i,j} \in A[[t]]$ 
and denote by $\Delta_{i,j} \in A[[t]]$  the minor of index $(i,j)$ of $\ug$.
We have $\ug^{-1}= \Bigl( t^{-Nd} \, \Delta_{i,j} \Bigr)_{i,j=1,..,N}$ so that the  
$(i,j)$--coefficient $D_{i,j, \sigma}$ of $g_{B[[t]]}^{-1} \,  \sigma(g_{B[[t]]})$.
is
\begin{equation}\label{formula01}
D_{i,j,\sigma}= \bigl(\frac{\sigma(t)}{t} \bigr)^d \, t^{-Nd} \, \sum\limits_{k=1}^N
\Delta_{i,k}(t) \,  P_{k,j}(\sigma(t)) \in B(\!(t)\!). 
\end{equation}
When $\sigma=1$, $C_{i,j,\sigma}$ specializes on  $\delta_{i,j}$ so that 
\begin{equation}\label{formula02}
D_{i,j,\sigma} =\delta_{i,j} \, + \, \bigl(\frac{\sigma(t)}{t} \bigr)^d \, t^{-Nd} \, \sum\limits_{k=1}^N \Delta_{i,k}(t) \,  
\bigl( P_{k,j}(\sigma(t)) - P_{k,j}(t) \bigr).
\end{equation}
It follows that there exists a uniform integer $w \geq 0$ such that
$g_{B[[t]]}^{-1} \sigma(g_{B[[t]]})$ belongs to $G(B[[t]])$
for $\sigma \in J^w(B)$.
\end{proof}

\medskip

Our construction defines then a $1$--cocycle $z(x) : J^{w(x)} \to L^+G$
for the Hochschild cohomology as defined by Demarche \cite[\S 2.1]{D} for the $A$--functor  
in groups  $J^{w(x)}$ and $LG$.
This induces an $1$-cocycle for the Hochschild cohomology  
$$
\Res(x): J^{w(x)} \to  G_{w(x)}.
$$
Since $J^{w(x)}$ acts trivially on $G_{w(x)}$, 
the map $\Res(x)$ is actually a homomorphism of $A$--groups.
This defines  classes
$\gamma(x)=[z(x)] \in H^1_{coc}(J^{w(x)}, LG)$ and
$[\Res(x)] \in H^1_{coc}(J^{w(x)}, G_{w(x)})$.

\begin{slemma} \label{lem_sorites2}
 Let $A'$ be a flat cover of $A$. Then $w(x_{A'(\!(t)\!)})=w(x)$
 and $\Res(x_{A'(\!(t)\!)})= \Res(x)$.
\end{slemma}

\begin{proof} We have the obvious relation $w':=w(x_{A'(\!(t)\!)}) \leq w(x)=:w$
(which holds in general). We
observe that $w=0$ implies that $w'=0$ as well.
We can assume $w \geq 1$.
By definition of  $w$, there 
exists a ring extension $B$ of $A$ and $\sigma \in J^{w-1}(B)$
such that $z_\sigma(x) \in G\bigl(B(\!(t)\!) \bigr)  \setminus G\bigl(B[[t]] \bigr)$.
We consider $B'=A' \otimes_A B$, this is a flat cover of $B$
so that $B$ injects in $B'$.   It follows that $B[[t]]= B'[[t]] \cap B(\!(t)\!)$.
Since $G$ is affine, we get that 
$(z_\sigma(x))_{B'}$ belongs to  $G\bigl(B'(\!(t)\!) \bigr)  \setminus G\bigl(B'[[t]] \bigr)$.
 Thus $w'>w-1$ and $w'=w$.
\end{proof}

\subsection{Compararison with the elementary construction}

For each integer $r \geq 1$,  
we consider the map of $\ZZ$--functors\footnote{It is not a group functor.}
$\phi^r: \GG_a \to J$,
which associates to the coordinate $u$ the element of $J(\ZZ[u])$ defined
by $t \mapsto t(1+ut^r)=t+ u \, t^{r+1}$. We observe that $\phi$ factorizes through
$J^{r+1}$.

\begin{slemma}\label{lem_phi}
The composite $\GG_a \xrightarrow{\phi^r} J^{r} \to J^{r}/J^{r+1}$
is a $\ZZ$--group isomorphism.
\end{slemma}

\begin{proof} We take two parameters $u_1,u_2$ and see that 
$\phi^{r}(u_1) \circ \phi^r(u_2) \in J(\ZZ[u_1,u_2])$ maps $t$ to
$$
 (t+ u_2 t^{r+1})+ u_1 (t+ u_2 t^{r+1})^{r+1} =
 t + (u_1 + u_2) t^r + (r+1) u_1 \, u_2 t^{r+2} + \dots
$$
\end{proof}

The following statement is then obvious.

\begin{slemma}\label{lem_obvious} For each integer 
$w \geq 1$ and each ring $B$, we have
$$
J^{w}(B)= \bigl\langle  \phi^r\bigl( B \bigr) \bigr\rangle_{r \geq w}.
$$
\end{slemma}

We consider firstly the case when the index of $x$ is integral.

\begin{slemma} \label{lem_comp} We assume that $r(x) \in \NN$.

\begin{numlist}
\item If $r(x) \geq 1$, then $w(x) = r(x)+1$ and 
the following diagrams
\[
\xymatrix{
 J^{w(x)} \ar[r]^{\Res(x)}  & G_{w(x)} \ar[d]    \quad & 
  J^{w(x)} \ar[d] \ar[r]^{\Res(x)}  & G_{w(x)} \ar[d] \\
  \GG_a \ar[u]^{\phi^{r(x)}} \qquad \ar[r]^{\qquad \res(x)} & G  \quad & 
   J^{w(x)}/  J^{w(x)+1} = \GG_a \qquad \ar[r]^{\qquad \res(x)} & G 
  }
\]
commute. Furthermore $\Res(x)$ is trivial on $J^{2 w(x)+1}$.

\item If $r(x) =0$, then $w(x)=0$ or $1$ and the following diagrams
\[
\xymatrix{
 J \ar[r]^{\Res(x)}  & G \quad & J \ar[d] \ar[r]^{\Res(x)}  & G \\
 \GG_m  \ar[u] \ar[ur]_{\res(x)} & \quad & \GG_m   \ar[ur]_{\res(x)}
}
\]
commute. 

\end{numlist}
 
\end{slemma}

\begin{proof} 
Once again we can localize for the flat topology and assume that 
the $G$--torsor is trivial.

\smallskip

\noindent  (1) Since the maps $\phi^{r}$ factorize by $L^+G$ for all $r \geq r(x)$,
it follows that $w(x) \leq r(x)$ according to the identity of Lemma
\ref{lem_obvious}.
On the other hand, $\phi^{r(x)-1}$ does not  factorize
by $L^+_G$, so that $w(x) > r(x)-1$. We conclude that 
$w(x)=r(x)$.
The commutativity of the left-hand side diagram is clear.
For the right hanside diagram we apply lemma \ref{lem_record}.(1)
to $r(g)$ and $M=1$. For $s \geq r(x)+1$, it follows that
$g^{-1} \, \sigma_s(g) \in \ker\Bigl( G\bigl(A^u[[t]]\bigr) \to 
 G\bigl(A^u[t]/t) \Bigr)$ so that 
 $$
 \phi^{s}(A) \subseteq  \ker\Bigl( J^{w(x)}(A)  \xrightarrow{\Res}
 G\bigl(A[t]/t^{w(x)+1} \bigr) \to G(A) \Bigr)
 $$
Lemma \ref{lem_obvious} yields that 
$$
 J^{w(x)+1}(A) \subseteq  \ker\Bigl( J^{w(x)}(A)  \xrightarrow{\Res}
 G\bigl(A[t]/t^{w(x)+1} \bigr) \to G(A) \Bigr).
 $$
Since we have the same property for each $A$--algebra $B$, 
 the right-hand side compatibility is established. 
The last assertion  follows from Lemma \ref{lem_record}.(1) applied to $M=r(x)+1$.
 \smallskip
 
\noindent (2) This is similar with (1) and the last assertion 
follows from Lemma \ref{lem_record}.(2).
\end{proof}

\begin{remarks}\label{rem_alternative}{\rm (a) Since $\Res$ is a group homomorphism,
Lemma \ref{lem_comp} shows that $\res$ is a group homomorphism.
We have then an alternative proof of the homomorphism part of Theorem \ref{thm_main}.(1) and (2) when the index
is integral. However we still need Proposition \ref{prop_index} for showing
that $\res(x)$ (and a fortiori $\Res(x)$)
are not trivial.

\smallskip

\noindent (b) In case (1), for $v=0,..., w(x)$,
the same argument shows more generally that   $\Res: J^{w(x)} \to G_{w(x)} \to G_{v}$
factorizes through $J^{w(x)}/ J^{w(x)+v+1}$.  
 }
\end{remarks}

\section{Loop groups and affine Grassmannians}

Our goal is to relate our construction with the theory of affine grassmannians.

\subsection{Ind-schemes}\label{subsec_ind}
We continue with the same setting
 with $G \subset \SL_N=:H$ a
closed $A$--subgroup scheme.
The loop groups are the $A$-functors
$$
B \, \mapsto  \,  L^+_G(B) = G(B[[t]]), \enskip
B \, \mapsto  \,  L_G(B) = G(B(\!(t)\!)).
$$
Both $A$--functors are equipped with the rotation action of $\GG_m$:
for each $A$--ring $B$, each $b \in B^\times$ and each
$g \in L_G(B)= G(B(\!(t)\!))$ we put $\delta(b).g=g(b^{-1} t)$.
The $A$--functor $L^+_G$ is representable by an $A$-scheme
and the $A$--functor $L_G$ is representable by an ind $A$-scheme.
More precisely $L_H$ is the union of the subfunctors
$L_{H,d}$ given by those matrices $A$ for which the entries
of $A$ are Laurent series of the form  $\sum\limits_{i \geq -d} a_i t^i$.
Each  $L_{H,d}$ is representable by an affine $A$--scheme
and we set $L_{G,d}=L_{H,d} \cap L_G$.

\smallskip

For $n \geq 1$, we denote by ${_nL}_G(B)= G\bigl(B(\!(t^{1/n})\!) \bigr)$
and similarly for $L_G^+$. 
As in the introduction, we consider the polynomial ring $B^u=B[u]$
seen as a subring of the Laurent polynomial ring $B[\lambda, \lambda^{-1}]$
by applying $u$ to $\lambda -1$.
We have a natural morphism
$L_G \to {_nL}_G$.
Let $r=m/n \in \QQ_{\geq 0}$ and consider the $A$--subfunctors of $L_G$
defined by $$
{^r\!L}_G(B)= \Bigl\{ [g] \in L_G(B)  \, \mid \,
 g^{-1} \sigma_r(g) \in G(B^u[[t^{1/n}]]) \Bigr\}.
$$
if $r >0$ and
$$
{^0\!L}_G(B)= \Bigl\{ [g] \in L_G(B)  \, \mid \,
 g^{-1} g(\lambda t) \in G(B[\lambda, \lambda^{-1}][[t^{1/n}]]) \Bigr\}.
$$
For $r>s \geq 0$, we have ${^s\!L}_G \subseteq {^r\!L}_G$ according to
Proposition \ref{prop_index} (resp.\, Lemma \ref{lem_0}) in the case $s>0$
(resp.\ $s=0$).

\begin{slemma}\label{lem_ind_var} (1) For each $d \in \ZZ_{\geq 0}$,
${^r\!L}_{G,d}= L_{G,d} \cap {^r\!L}_G$ is a closed $A$--subscheme of $L_{G,d}$
and even of $L_{G,d/N}$.

\smallskip

\noindent (2) ${^r\!L}_G$ is a closed ind-subscheme of $L_G$.

\end{slemma}

\begin{proof} (1) Once again one can work with $H=\SL_N$.
\smallskip
\noindent{\it First case: $r>0$:}
For each $d \geq 0$, we have by definition
$$
{^r\!L}_{H,d}(B) =  \Bigl\{ g\in L_{H,d}(B)  \, \mid \,
 g^{-1} \sigma_r(g) \in \SL_N(B^u[[t^{1/n}]])={_nL}_{H,d}(B^u) \Bigr\}.
$$
We consider the $\mathbf{A}^1_A$-morphism
$\psi_{r,d}:
L_{H,d} \times_A \mathbf{A}^1_A \to  {_nL}_{H,nNd} \times_A \mathbf{A}^1_A$
defined
by $(g,u) \mapsto  g^{-1} \sigma_r(g)$.
The $A$--functor ${^r\!L}_{H,d}$ is the $A$--subfunctor of
$L_{H,d}$ consisting in the $g \in L_{H,d}$ such that
the map $u \mapsto  g^{-1} \sigma_r(g)=\psi_{r,d}(g)$ factorizes through
the closed $A$-subscheme ${_nL}_{H,0} \times_A \mathbf{A}^1_A$.
Since $\GG_{a,A}$ is free over $A$ (i.e. $A[t]$ is a free $A$--module),
it follows that the $A$--functor ${^r\!L}_{H,d}$ is representable by
a closed $A$--subscheme of $L_{H,d}$ according to \cite[VI$_B$.6.2.4]{SGA3}.
According to Corollary \ref{cor_main}, for each $A$--field $E$,
we have ${^r\!L}_{H,d}(E) \subseteq {L}_{H,N/d}(E)$
so that ${^r\!L}_{H} \subseteq {L}_{H,N/d}$.

\smallskip

\noindent{\it  Case $r=0$}. The preceding argument
works verbatim with the morphism \break
$\psi_{d}: {L_{H,d} \times_A \GG_{m,A}} \to  L_{H,Nd} \times_A \GG_{m,A}$,
$(g,\lambda) \,  \mapsto  \, g^{-1} g(\lambda t)$.

\smallskip

\noindent (2) This is a straightforward consequence of (1).
\end{proof}

\begin{sremark}{\rm The ${^r\!L}_G$'s have no reason to be
$A$--schemes and this happens already for $\GG_m$
since ${^0\!L}_{\GG_m}={L}_{\GG_m}$.
}
\end{sremark}

By specialization at $t=0$, we get the residue $A$--functors
$$
{^r\!\res}_G : {^r\!L}_{G} \to \Hom_{gr}(\GG_a,G)  \enskip (r>0),
{^0\!\res}_G :  {^0\!L}_{G} \to \Hom_{gr}(\GG_m,G).
$$

\medskip

By definition, the affine grassmannian $\cQ_G$ of $G$ is the
fppf sheafification of the $A$--functor
$$
B \, \mapsto  \,  \cF(B) = G(B(\!(t)\!))/ G(B[[t]]).
$$
Let $r \in \QQ_{\geq 0}$ and consider the $A$--subfunctor of $\cF$
defined by
$$
 \cF^{r}(B) =  \Bigl\{ [g] \in {^r\!L}_G(B) / L_G^+(B) \Bigr\}.
$$
It is an $A$--subfunctor granting to Lemma \ref{lem_sorites}.(2)
and we denote its fppc sheafification by ${^r\!\cQ}_G$.
Clearly the map $LG \to \cQ_G$ induces
an isomorphism $({^r\!L}_G/L^+_G)_{fppf} \simlgr {^r\!\cQ}_G$.
Also the residue $A$--functors give rise to
the $A$--functors
$$
{^r\!\res}_G : {^r\!\cQ}_{G} \to \overline{\Hom_{gr}}(\GG_a,G)  \enskip (r>0),
{^0\!\res}_G :  {^0\!\cQ}_{G} \to \overline{\Hom_{gr}}(\GG_m,G);
$$
where $\overline{\Hom_{gr}}(\GG_a,G)$ is the fppf quotient 
$\Hom_{gr}(\GG_a,G)/G$ and similarly for $\GG_m$. 

\begin{slemma} \label{lem_fixed} ${^0\!\cQ}_G$ is the 
fixed locus for the rotation action of $\GG_m$  on $\cQ_G$.
\end{slemma}

\begin{proof} The fixed locus  ${^\sharp \! \cQ}_G$ is the $A$--subfunctor of
 $\cQ_G$ defined by
  $$
  {^\sharp \! \cQ}_G(B)= \Bigl\{ x \in \cQ_G(B) \, \mid \, x(c t)=
  x(t) \enskip \forall \hbox{\enskip $B$--algebra $C$ and \enskip} \forall c \in
  C^\times \Bigr\}.
 $$
 Clearly ${^0\! \cQ}_G$ is an $A$--subfunctor of ${^\sharp \! \cQ}_G$.
 Conversely we are given an $A$--algebra $B$ and
 an element $x \in {^\sharp \! \cQ}_G(B)$. To show that
 $x$ is fixed is local for the fppf topology so that
 we may assume that $x$ lifts to an element $g \in L_G(B)$.
 We take $C=B[\lambda^{\pm 1}]$ so that $x(\lambda t)= x(t) \in \cQ_G(C)$
 It follows that $g(\lambda t) = g(t) h$ with $h \in L_G^+(C)$.
 hence  $g^{-1} \, g(\lambda t) \in G\bigl( B[\lambda^{\pm 1}][[t]] \bigr)$.
 In other words $g$ belongs  to ${^0\!L}_G(B)$ and $x \in {^0\! \cQ}_G(B)$.
\end{proof}

\subsection{Ind-schemes, II}
Now we assume that the quotient $\SL_N/G$ is representable
by an quasi--affine $A$--scheme.
In this case, the structure of ind $A$--scheme of $\cQ_{\SL_N}$
induces a structure of ind $A$--scheme on $\cQ_{G}$
such that the map $\cQ_G \to \cQ_{\SL_N}$ is a locally closed immersion
(which is closed if $\SL_N/G$ is affine), see \cite[lemma 2.14]{Go}.

\begin{sproposition}
 ${^r \!\cQ}_G$ is a closed  $A$--ind-subscheme of $\cQ_G$.
\end{sproposition}

\begin{proof}
 It is enough to consider the case of $\SL_N$.
 For each $d \geq 0$,
 the map $\SL_{N,d} \to \cQ_{H,d}$ is a $L_H^+$--torsor (locally trivial
 for the \'etale topology).
 It follows that the quotient sheaf
 ${^r\!\SL}_{N,d} / L_H^+$ is representable
 by a closed $A$-subscheme of the $A$-scheme $\cQ_{H,d}$.
 Thus ${^r \!\cQ}_G$ is a closed  $A$--ind-subscheme of $\cQ_G$.
\end{proof}

\subsection{Case of a split reductive $k$--group}
We work here over a base field $k$ and assume that $G$ is a
split reductive $k$--group equipped with a Killing couple
$(B,T)$. We denote by $\Phi=\Phi(G,T)$ the associated
root system and by $\Delta$ the subset of simple roots
with respect to $B$.
For each $\alpha \in \Phi$, we denote by $U_\alpha$ the image of 
the root homomorphism $u_\alpha: \GG_a \to
G$ attached to $\alpha$.
We shall use also the $k$--group $L^{<0}_G$ defined by 
$L^{<0}_G(R)= G\bigl(R[\frac{1}{t}]\bigr)$ for each $k$--algebra $R$.

We  consider the cell $L_G^\mu = L_G^+ \,  t^\mu \, L_G^+$ where $\mu$ 
is a non-negative coweight $\mu: \GG_m \to T$ and the corresponding  affine Schubert cell
$\cQ_\mu = (L_G^+ \, t^\mu\, L_G^+)/ L_G^+$ of $\cQ$. We know 
that $\cQ_\mu$ is a smooth $k$--variety.
Also the map $s_0: L_G^+ \to G$ induces a $G$--equivariant map
$$
p_\mu: \cQ_\mu \to G/P_G(-\mu)
$$
where $P_G(-\mu)$ stands for the Richardson
parabolic subgroup attached to $\mu$ \cite[\S 2]{NP}.
The orbit map $G \to \cQ_\mu$, $g \mapsto  \, g \,  t^\mu$ induces
a closed immersion  $i_\mu:  G/P_G(-\mu) \to \cQ_\mu$ which
is a section of $p_\mu$.

We consider the $k$--groups $L^{\geq \mu} G= t^\mu L^+_G t^{-\mu}$
and $L^{<\mu} G= t^\mu L^{<0}_G t^{-\mu}$.
Let $J= L^+_G \cap s_0^{-1}(B)$ 
and define $J^{\geq \mu}= J \cap L^{\geq \mu}_G$
and $J^\mu=  J \cap L^{< \mu}_G$.

According to \cite[lemme 2.2]{NP}, we have an isomorphism
$J^\mu \times J^{\geq \mu} \simlgr  J$ and the 
map $J^\mu \to \cQ_\mu$, 
$g \to g \, t^\mu$ is an open immersion; we denote by $\Omega^\mu$ its image.
We choose a total order on $\Phi_+$; furthermore the product in $L^+_G$ induces
an isomorphism
(of $k$--varieties)

\[
\xymatrix{
j_\mu : & \prod\limits_{ \alpha \in \Phi
\mid \langle \mu , \alpha \rangle \geq 1 }
\prod\limits_{ i=0}^{\langle \mu , \alpha \rangle} U_{\alpha,i}
  & \simlgr    &  J^\mu
  }
\]
where $U_{\alpha,i}$ is the image of $u_{\alpha,i}: \GG_a\to L_G^+$,
$x \mapsto u_{\alpha}(t^i x)$.

\begin{slemma} \label{lem_bruhat} The following diagram 

\[
\xymatrix{
\prod\limits_{ \alpha \in \Phi
\mid \langle \mu , \alpha \rangle \geq 1 }
\prod\limits_{ i=0}^{\langle \mu , \alpha\rangle} U_{\alpha,i} \ar[d] & \simlgr    &  
J^\mu \enskip  \ar@{^{(}->}[r] & \cQ_\mu \ar[d]^{p_\mu} \\
\prod\limits_{ \alpha \in \Phi
\mid \langle \mu , \alpha \rangle \geq 1 }
\GG_{a} & \simlgr& U_G(\mu) \enskip  \ar@{^{(}->}[r] & G/P_G(-\mu) 
}
\]
commutes and is cartesian where the bottom map is induced by 
the embedding $U_G(\mu) \hookrightarrow G$ and the left vertical map is the projection on
the factors involving $i=0$.
 
\end{slemma}

Note that  $U_G(\mu) \to G/P_G(-\mu) $ is an open embedding (Bruhat big cell).

\begin{proof} The commutativity is obvious. We have the 
 the inclusions $\Omega_\mu \subseteq p_\mu^{-1}(U_G(\mu).[1]) \subset \cQ_\mu$
 and we have to show that $\Omega_\mu \subseteq p_\mu^{-1}(U_G(\mu).[1])$.
 
Up to extend $k$ to its algebraic closure, 
it enough to show that $p_\mu^{-1}(U_G(\mu).[1])(k) \subseteq \Omega_\mu(k)$.
We are given an element $g.t^\mu \in  p_\mu^{-1}(U_G(\mu).[1])(k)$
with $g \in L^+_G(k)$. Since $g(0).[1]$ belongs in the big Bruhat cell
$V= U_G(\mu).[1]$ of $G/ P_G(-\mu)$, it follows that 
$g.[1] \in V(k[[t]])$.
We get that $g \in  U_G(\mu)(k[[t]]) \,  P_G(-\mu)(k[[t]])$.
We can assume that $g \in U_G(\mu)(k[[t]])=L_{U_G(\mu)}^+(k)$.
The decomposition $J^\mu \times J^{\geq \mu} \simlgr  J$
induces a decomposition ${J^\mu \times \bigl(J^{\geq \mu} \cap L_{U_G(\mu)}^+\bigr) } \simlgr  L_{U_G(\mu)}^+$.
Since $J^{\geq \mu}$ fixes $t^\mu$, we get that 
$g.t^\lambda \in \Omega_{\mu}(k)$. 
\end{proof}

We get then an  an isomorphism

\[
\xymatrix{
 J^\mu_+:=\prod\limits_{ \alpha \in \Phi
\mid \langle \mu , \alpha \rangle \geq 1 }
\prod\limits_{i=1}^{ \langle \mu , \alpha \rangle}
U_{\alpha,i}  & \simlgr   & 
p_\mu^{-1}([1]) \enskip \subset \enskip  \cQ_\mu}.
\]
We put  $E^\mu=p_\mu^{-1}([1])$, it is an affine 
$k$--space equipped with a left action of $P_G(-\mu)$.
We get then a $P_G(-\mu)$-equivariant map
$G \times_k E^\mu \to \cQ_\mu$ where $G$ is equipped with the
left action provided by right translations.
The next statement is well-known.

\begin{slemma} The quotient $G \wedge^{P_G(-\mu)} E^\mu$
is representable by the smooth $k$-variety $\cQ_\mu$.
\end{slemma}

\begin{proof} The action $P_G(-\mu)$ on  $G \times E^\mu$  (and on $G$) is free.
We apply Lemma \ref{lem_artin} to  the $P_G(-\mu)$-morphism
$G \times_k E^\mu \to G$ and get that the fppf quotient  $G \wedge^{P_G(-\mu)} E^\mu$
is representable by a $k$--scheme; we observe that this $k$--scheme 
is of finite type according to the permanence properties \cite[$_2$.2.7.1]{EGA4}.
Similarly,  using \cite[$_4$.17.7.3]{EGA4}, we 
see that $G \wedge^{P_G(-\mu)} E^\mu$ is smooth over $G/P_G(-\mu)$.

By construction, the map  $G/P_G(-\mu)$-map $G \wedge^{P_G(-\mu)} E^\mu \to \cQ_\mu$ 
is an isomorphism over all geometric fibers over  $G/P_G(-\mu)$.
According to Grothendieck's fiberwise isomorphism
criterion \cite[$_4$.17.9.5]{EGA4}, we conclude that $Z \to \cQ_\mu$ 
is an isomorphism.

\end{proof}

\begin{sremarks}{\rm
 (a) If $\mu$ is minuscule, then the map  $p_\mu$ is an isomorphism.

 \smallskip

 \noindent (b) If $\mu$ is quasi-minuscule (i.e. minimal
 but not minuscule), $p_\mu$ is the line bundle of
  Ng\^o-Polo \cite[\S 7]{NP}.
 }
\end{sremarks}

\begin{sproposition}\label{prop_fixed0} \cite[after 2.1.11]{Zhu}
The morphism $i_\mu$ induces an isomorphism \break
$G/P_G(-\mu) \simlgr  {^\sharp\!\cQ}_\mu$ on the fixed point locus
for the rotation action. 
\end{sproposition}

\begin{proof} Since $\cQ_\mu$ is smooth, 
the $k$-subvariety ${^\sharp\!\cQ}_\mu$ is smooth as well.
It is then enough to check that $i_\mu$ induces 
an isomorphism $\bigl( G/P_G(-\mu) \bigr)(\ol k) 
\simlgr {^\sharp\!\cQ}_\mu(\ol k)$.
Let $q \in  {^\sharp\!\cQ}_\mu(\ol k)$.
Up to conjugate by an element of $G(k)$ we can
assume that $p_\mu(x)=1$ so that $x \in E^\mu(k)$.
But ${^\sharp\!E_\mu}(k)=\{  p_\mu(0) \}$ so that $x =i_\mu([1])$.
\end{proof}

Taking into account Lemma \ref{lem_fixed} we get an
isomorphism $G/P_G(-\mu) \simlgr {^0 \!\cQ}_\mu$.

\begin{sproposition}\label{prop_fixed} Assume that $\mu >0$. 
Let $g \in L_G^\mu(R)$ where $R$ is a $k$--algebra.

\smallskip

\noindent (1) If $R$ is  semilocal, 
then $r(g)=0$ if and only if $g \in G(R) \, t^\mu L_G^+(R)$. 

\smallskip 

\noindent (2) Assume that $r(g)=0$ and denote by $x$ the image 
of $g$ in $G/P_G(-\mu)(R) \simlgr {^0\!\cQ_\mu}(R)$.
It defines a $P_G(-\mu)$--torsor $E(g)$. 
Then $E(g)$ is a trivial $P_G(-\mu)$--torsor if and only if
$g \in G(R)\, t^\mu L_G^+(R)$.

\end{sproposition}

It follows of (1) that the strata $\cQ_\mu$ which contains $[g]$
is encoded in $\res(g)$.

\begin{proof}
(1)  If $g \in G(R)\, t^\mu L_G^+(R)$, then $r(g)=r(t^\mu)=0$.
Conversely we assume that $r(g)=0$.
We denote by $x \in {^0\!\cQ_\mu}(R)$
the  image of $g$. Since $G/P_G(-\mu) \simlgr {^0\!\cQ_\mu}$ 
and $R$ is semilocal there exists $h \in G(R)$ such that 
$x=[h.t^\mu]$ \cite[XXVI.5.10.(i)]{SGA3}. We conclude  that $g \in  G(R) \,  t^\mu L_G(R)$.  
 
 \smallskip
 
\noindent (2) If $g =h \, t^\mu \, h' \in G(R)\, t^\mu L_G^+(R)$, then 
$x=[h] \in (G/P_G(-\mu))(R)$. We have $E(g)(R) \not=\emptyset$
so that the $P_G(-\mu)$--torsor $E(g)$ is trivial.
Conversely we assume that $E(g)$ is the trivial $P_G(-\mu)$--torsor.
Then $E(g)(R) \not=\emptyset$, that is there exists $h \in G(R)$
such that $x=[h]$. It follows that $h t^\mu=g \in L_G(R)/L_G^+(R)$
whence $g \in G(R)\, t^\mu L_G^+(R)$.
\end{proof}

\begin{sproposition}\label{prop_lattice} 
\noindent (1) The ind $k$-variety ${^\sharp \!\cQ}$ is a $k$--scheme and 
we have ${^\sharp \!\cQ} =  \bigsqcup\limits_{ \mu \geq 0}  {^0\!\cQ}_\mu$.

 \smallskip
 
\noindent (2) Let $R$ be a semi-local connected $k$--algebra.
We have $${^0\!L}_G(R)= \bigsqcup\limits_{ \mu \geq 0}  G(R) \, t^\mu \, L_G^+(R) .$$

\end{sproposition}

\begin{proof}
(1) We write $\cQ= \limind_{\theta  \geq 0} \overline{\cQ}_\theta$ as an inductive limit of
projective varieties. Let $\theta$ be a non-negative coweight.
The $\cQ_\mu$'s for $0\leq \mu \leq \theta$ provide a stratification of 
$\overline{\cQ}_\theta$.
We consider the  map $\psi_\theta: \bigsqcup\limits_{0 \leq \mu \leq \theta}
{^0\!\cQ}_\mu \to   
{^\sharp\!\overline{\cQ}}_\theta$ and claim that it is an isomorphism.
The left-hand side is a projective variety and so is the right handside. 
Each piece ${^0\!\cQ}_\mu$ is a closed (smooth) subvariety of  $\overline{\cQ}_\theta$.
Also  ${^0\!\cQ}_\mu \cap {^0\!\cQ}_{\mu'} =\emptyset$ for $\mu ' \not = \mu$
 since the residue encodes the strata.
 It follows that $\psi_\theta$ is a closed immersion. 
 On the other hand, the map $\bigsqcup\limits_{0 \leq \mu \leq \theta} 
 {^0\!\cQ}_\mu(\ol k) \to   
{^\sharp\!\overline{\cQ}}_\theta(\ol k)$ is bijective so that 
$\psi_\theta$ is a bijective closed immersion.
If follows that $^\sharp\!\overline{\cQ}_\theta = 
\bigsqcup\limits_{0 \leq \mu \leq \theta} Z_{\theta, \mu}$
where $Z_{\theta, \mu}$ is the unique connected component of
${^\sharp\!\overline{\cQ}}_\theta$
containing the image of ${^0\!\cQ}_{\mu}$ for each $\mu$.
We observe that  $Z_{\theta, \theta}$ is closed in ${^\sharp\!\overline{\cQ}}_\theta$
and does not intersect the boundary ${^\sharp\!\overline{\cQ}}_\theta \setminus \cQ_\theta$
so that    $ {^\sharp\!\cQ}_\theta \subseteq  Z_{\theta, \theta} \subset \cQ_\theta$.
Thus  ${^\sharp\!\cQ}_\theta =  Z_{\theta, \mu}$ by taking the invariants under
the rotation action.

We oberve that $Z_{\theta, \mu}= Z_{\theta',\mu}$ for all coweights
satisfying $0 \leq \leq \theta \leq \theta'$
so that ${^\sharp\!\cQ}_\mu= Z_{\theta, \mu}$  
for all coweights
satisfying $0 \leq  \mu \leq \theta$.
Thus $\psi_\theta$ is an isomorphism for all non-negative
coweights $\theta$.

Passing to the limit on $\theta$ yieds the wished statement.

\smallskip

\noindent (2) We combine (1) and Proposition \ref{prop_fixed}.
\end{proof}

\medskip

Given an element $g$ of $L_G^\mu(k)$, we would like to investigate
in a few cases its index and its residue. We write it
$g=g_0\,  \Bigl(\prod\limits_{ \alpha \in \Phi \mid \langle \mu , \alpha \rangle \geq 2 }
u_{\alpha}\bigl( \sum\limits_{i=1}^{\langle \mu , \alpha \rangle -1} t^i x_{\alpha,i}   \bigr)
\Bigr) \enskip t^\mu$. Then $g$ and
$\Bigl(\prod\limits_{ \alpha \in \Phi \mid \langle \mu , \alpha \rangle \geq 2 }
u_{\alpha}\bigl( \sum\limits_{i=1}^{\langle \mu , \alpha \rangle -1} t^i x_{\alpha,i}   \bigr)
\Bigr) \enskip t^\mu$ have same index and conjugated residues,
so that we may work with
$g= \Bigl(\prod\limits_{ \alpha \in \Phi \mid \langle \mu , \alpha \rangle \geq 2 }
u_{\alpha}\bigl( \sum_{i=1}^{\langle \mu , \alpha \rangle -1 } t^i x_{\alpha,i}   \bigr)
\Bigr) \enskip  t^\mu$.

\noindent{\it The minuscule case.} We have $g= t^\mu$ hence $r(g)=0$ and
$\res(g)=t^\mu$.

\noindent{\it The quasi-minuscule case.} Let $\gamma$ be the unique root
 satisfying $\langle \mu ,  \gamma \rangle \geq 2$.
We have
$g= u_\gamma\bigl(  t x \bigr) \, t^\mu$ for some  $x \in k$.
For $r=m/n$, it follows that

\begin{eqnarray} \nonumber
g^{-1} \, \sigma_r(g) &= & t^{-\mu} \, u_{\gamma}(-tx) \,
u_{\gamma}( t(1+ut^r)x)\,  t^\mu \bigl( \sigma_r(t)/t\bigr)^\mu \\ \nonumber
&=&t^{-\mu} \, u_{\gamma}\bigl( u t^{r+1} x \bigr) \, t^{-\mu} \, 
\bigl( \sigma_r(t)/t\bigr)^\mu \\ \nonumber
&=& u_{\gamma}\bigl( u t^{r+1-\langle \mu ,  \gamma \rangle  } x \bigr) \, t^{-\mu} \,  \bigl( \sigma_r(t)/t\bigr)^\mu.
\end{eqnarray}

\noindent It follows that $r =\langle \mu ,  \gamma \rangle -1$ 
if $x \not =0$.

\smallskip

\noindent{\it The rank one case.}
There exists a unique  root $\alpha $ satisfying
$\langle \mu , \alpha \rangle \geq 1$ and
we assume that $\langle \mu , \alpha \rangle \geq 2$ (since the case
$\langle \mu , \alpha \rangle =1$ is minuscule).
We have
$g= \bigl(  P(t) \bigr) \, t^\mu$ for some polynomial $P \in k[t]$
of degree $\leq  \langle \mu , \alpha \rangle -1$ satisfying $P(0)=0$.
For $r=m/n$, it follows that
$$
g^{-1} \, \sigma_r(g)=  t^{-\mu} \, u_{\alpha}(-P(t)) \,
u_{\alpha}( P(t(1+ut^r))\,  t^\mu \bigl( \sigma_r(t)/t\bigr)^\mu
=u_{\alpha}\Bigl(  \frac{P(t)- P(t(1+ut^r))}{t^{\langle \mu , \alpha \rangle}} \Bigr) \, \bigl( \sigma_r(t)/t\bigr)^\mu.
$$
It follows that
$$
r \geq r(g)  \enskip \Longleftrightarrow   \enskip
\frac{P(t)- P(t(1+ut^r))}{t^{\langle \mu , \alpha \rangle}} \in k^u[[t^{1/n}]].
$$
For example, for $P(t)=t$, we find that $r(g)= \langle \mu , \alpha \rangle$
and that $\res(g)=u_\alpha$.

\begin{scorollary}
With the preceding notations we  have $r(g)=0$ if and only if $P=0$.
\end{scorollary}

Note that this is coherent 
with Proposition \ref{prop_fixed}.

\begin{proof} If $P=0$, we have $g=t^\mu$ so that $r(g)=0$.
Conversely, we assume that $r(g)=0$ so that
$g^{-1} \, g(\lambda t) \in G\bigl( k[\lambda^{\pm 1}][[t]] \bigr)$.
The above computation with $\lambda=1+ut^0$ shows that
$$
g^{-1} \, g(\lambda t)
=u_{\alpha}\bigl( \frac{P(t)- P(\lambda t)}{t^{\langle \mu , \alpha \rangle}} \bigr)
\, \mu(\lambda)  \in G\bigl( k[\lambda^{\pm 1}][[t]] \bigr)
$$
Since $P$ is of degree $<\langle \mu , \alpha \rangle$,
we conclude that $P=0$.
\end{proof}

\medskip

\section{Appendix: descent}

\begin{slemma}\label{lem_descent}
Let  $S$ be a scheme. Let  $G$ be a  $S$--group scheme flat of finite type.
We are given a left action of $G$ on an $S$-scheme 
 $X$ of finite type such that $X$ admits a  $G$--linearized line 
 bundle $\mathcal{L}$ which is relatively ample over  $S$. 
Let  $T$ be an  $S$--scheme and let   $f: E \to T$ be a  $G_T$-torsor.
Then the contracted product  $E \wedge^G X_T$ over $T$ is representable 
by a  $T$--scheme. Furthermore we have

\begin{romlist}
 \item the line bundle
$\mathcal{M}=E \wedge^G \mathcal{L}$
on  $E \wedge^G X_T$  is relatively ample over   $T$.
 \item If  a  $S$-group scheme   $J$ acts  (on the left) on  $f: E \to T$
 such that   $T$ admits a $J$-linearized line bundle  which is  relatively ample
 over $S$, then   $E \wedge^G X_T$  admits a $J$-linearized line bundle
 which is  relatively ample  over $T$.
\end{romlist}

\end{slemma}

\medskip

\begin{proof}
The first part with property (i)  is  \cite[\S 10.2, lemme 6]{BLR}.
For establishig (ii), we are given a  $J$--linearized line bundle
$\mathcal{N}_0$ on  $T$ which is relatively ample over  $S$.
We consider the mapping $h: E \wedge^G X_T \to T$, 
we know that there exists a positive integer  $n$ such that 
the line bundle $\mathcal{N}=   h^*(\mathcal{N}_0)^{\otimes n}  \otimes 
\mathcal{M} $  on  
$E \wedge^G X_T \to T$ is relatively ample over $T$ \cite[4.6.13.(ii)]{EGA2}.
This line bundle is $J$--linearized as desired.
\end{proof}

\begin{slemma} \label{lem_artin} Let $S$ be a scheme and let $G$ be a flat  
$S$-group scheme locally of finite 
type. We are given a $G$--morphism of $S$-schemes $f: X \to Y$.
We assume that $G$ acts freely on $X$ and on $Y$, that $f$ is affine
and that the fppf quotient $Y/G$ is representable by a 
$S$--scheme. Then the  fppf quotient $X/G$ is representable
by an $S$-scheme.
\end{slemma}

\begin{proof}
Put $Z:=Y/G$.
Assume first that the $G$-torsor $Y \lra Z$ is trivial. Equivalently, the $G$-scheme $Y$ is isomorphic, over $S$, to $G \times_S Z$. Choosing such an isomorphism allows to consider $f$ as a $G$-morphism $X \lra G \times_S Z$. Put $X_0:=f^{-1}( {e} \times_S Z)$; it is a closed subscheme of $X$, which is transverse to the $G$-action: the natural $G$-morphism $G \times_S X_0 \lra X$ is an isomorphism. Thus, $X/G$ is represented by $X_0$.\\
For the general case, we proceed by descent. Considering $f$ as a morphism of $Z$-schemes, we may replace $S$ by $Z$ (and $G$ by $G \times_S Z$), and assume for simplicity that $Z=S$. Hence, $Y$ is a $G$-torsor over $S$. Put $S':=Y$. We are going to base-change the situation via the morphism $S' \lra S$. Put \[ f':=f \times_S S': X':=X \times_S S' \lra Y':=Y \times_S S'.\] Now, the $G$-torsor $Y' \lra S'$ is trivial, so that the quotient $X' \lra X'/G$ exists (and is a trivial $G$-torsor) by the discussion above. Since $X' \longrightarrow Y'$ is affine,  $X'/G \lra Y'/G=S'$ is affine as well.  It is equipped with a canonical descent data for the fpqc morphism $S' \lra S$. Hence, this data is effective (by fpqc descent for affine schemes), yielding an arrow $\tilde X \lra S$. Now, using descent for morphisms, the $S'$-arrow  $X' \lra X'/G$ descends to an $S$-arrow $X \lra \tilde X$. This is the quotient $X \lra X/G$ we sought for. This fact can, again, be checked by descent.
\end{proof}

\bigskip

\medskip

\end{document}